\documentclass[12pt]{article}
\usepackage{}

\usepackage{amssymb}   %basic
\usepackage{amsthm}    %theoremstyle
\usepackage{amsmath}   %\mathcal
\usepackage{stmaryrd}  %\interleave
\usepackage{titletoc}  %\ttl
\usepackage{mathrsfs}  %\mathscr
\usepackage{graphicx}

\newlength{\defbaselineskip}
\newcommand{\setlinespacing}[1]%
           {\setlength{\baselineskip}{#1 \defbaselineskip}}

\theoremstyle{plain}
\newtheorem{thm}{Theorem}[section]
\newtheorem{cor}[thm]{Corollary}
\newtheorem{lem}[thm]{Lemma}
\newtheorem{prop}[thm]{Proposition}

\theoremstyle{definition}
\newtheorem{defn}{Definition}[section]

\newtheorem{rmk}{Remark}[section]

\newcommand{\eps}{\varepsilon}

\DeclareMathOperator*{\esssup}{ess\,sup}

\newcommand{\cO}{\mathcal{O}}

\newcommand{\cL}{\mathcal{L}}
\newcommand{\cM}{\mathcal{M}}
\newcommand{\cB}{\mathcal{B}}
\newcommand{\cS}{\mathcal{S}}

\newcommand{\cV}{\mathcal{V}}

\newcommand{\bP}{\mathbb{P}}
\newcommand{\bR}{\mathbb{R}}
\newcommand{\bN}{\mathbb{N}}

\newcommand{\sF}{\mathscr{F}}
\newcommand{\sP}{\mathscr{P}}

\newcommand{\sU}{\mathscr{U}}
\newcommand{\sV}{\mathscr{V}}

\newcommand{\tbf}{\textbf}

\makeatletter\@addtoreset{equation}{section} \makeatother
\begin{document}

\title{ Maximum Principle for Quasi-linear Backward Stochastic Partial Differential Equations\footnotemark[1] }

\author{Jinniao Qiu \footnotemark[2] ~~~~ and ~~~~
Shanjian Tang\footnotemark[2]~\footnotemark[3]}
%\date{}

\footnotetext[1]{Supported by NSFC Grant \#10325101, by Basic Research Program of China (973 Program)  Grant \# 2007CB814904, by the Science
Foundation of the Ministry of Education of China Grant \#200900071110001, and by WCU (World Class University) Program through the Korea
Science and Engineering Foundation funded by the Ministry of Education, Science and Technology (R31-2009-000-20007).}

\footnotetext[2]{Department of Finance and Control Sciences, School of Mathematical Sciences, Fudan University, Shanghai 200433, China.
\textit{E-mail}: \texttt{071018032@fudan.edu.cn} (Jinniao Qiu), \texttt{sjtang@fudan.edu.cn} (Shanjian Tang).}

\footnotetext[3]{Graduate Department of Financial Engineering, Ajou University, San 5, Woncheon-dong, Yeongtong-gu, Suwon, 443-749, Korea.}

\maketitle

%----------------------
\begin{abstract}
 In this paper we are concerned with the maximum principle for quasi-linear backward stochastic partial differential equations (BSPDEs for short)
 of parabolic type. We first prove the existence and uniqueness of the weak solution to quasi-linear BSPDE with the null Dirichlet condition on
 the lateral boundary. Then using the  De Giorgi iteration scheme, we establish the maximum estimates
 and the global  maximum  principle for quasi-linear BSPDEs.
 To study the local regularity of weak solutions, we also prove a local maximum principle for the backward stochastic parabolic De Giorgi class.
\end{abstract}

AMS Subject Classification: 60H15; 35R60

Keywords: Stochastic partial differential equation,
Backward stochastic partial differential equation, De Giorgi iteration, Backward stochastic parabolic De Gigorgi class

\section{Introduction}
In this paper we investigate the following quasi-linear BSPDE:
\begin{equation}\label{1.1}
  \left\{\begin{array}{l}
  \begin{split}
  -du(t,x)=\,&\displaystyle \biggl[\partial_{x_j}\Bigl(a^{ij}(t,x)\partial_{x_i} u(t,x)
        +\sigma^{jr}(t,x) v^{r}(t,x)     \Bigr) +b^j(t,x)\partial_{x_j}u(t,x)\\
        &\displaystyle
         +c(t,x)u(t,x)+\varsigma^{r}(t,x)v^r(t,x)+g(t,x,u(t,x),\nabla u(t,x),v(t,x))\\
        &\displaystyle +\partial_{x_j}f^j(t,x,u(t,x),\nabla u(t,x),v(t,x))
                \biggr]\, dt\\ &\displaystyle
           -v^{r}(t,x)\, dW_{t}^{r}, \quad
                     (t,x)\in Q:=[0,T]\times \mathcal {O};\\
    u(T,x)=\, &G(x), \quad x\in\cO.
    \end{split}
  \end{array}\right.
\end{equation}
Here and in the following we use Einstein's summation convention, $T\in(0,\infty)$  is a fixed deterministic terminal time, $\cO\subset \bR^n$ is a
bounded domain with $\partial \cO \in C^1$, $\nabla=(\partial_{x_1},\cdots,\partial_{x_n})$ is the gradient operator and $(W_t)_{t\in [0,T]}$ is an
$m$-dimensional standard Brownian motion in the filtered probability space $(\Omega,\sF,(\sF_t)_{t\geq 0},P)$. A solution of BSPDE (\ref{1.1})
is a pair of random fields $(u,v)$ defined on $\Omega\times[0,T]\times\cO$ such that   (\ref{1.1}) holds in a weak sense (see Definition
\ref{definition of weak solution}).

The study of backward stochastic partial differential equations (BSPDEs)  can be dated back  about thirty years ago (see Bensoussan
\cite{Bensousan_83} and Pardoux \cite{Pardoux1979}). Such BSPDE arises in many applications of probability theory and stochastic processes,
for instance in the nonlinear filtering and stochastic control theory for processes with incomplete information, as an adjoint equation of the
Duncan-Mortensen-Zakai filtration equation (for instance, see \cite{Bensousan_83,Hu_Ma_Yong02,Hu_Peng_91,Tang_98,Zhou_92,Zhou_93}). In the
dynamic programming theory, some nonlinear BSPDEs as the so-called  backward stochastic Hamilton-Jacobi-Bellman equations, are also introduced
in the study of non-Markovian control problems (see Peng \cite{Peng_92} and Englezos and Karatzas \cite{EnglezosKaratzas09}).

The maximum principle is a powerful tool to study the regularity of solutions, and constitutes a beautiful chapter of the classical theory of
deterministic second-order elliptic and parabolic partial differential equations. Using the technique of Moser's iteration, Aronson and Serrin
proved the maximum principle and local bound of weak solutions  for deterministic quasi-linear parabolic equations (see~\cite[Theorems 1 and
2]{AronsonSerrin1967}), which are stated in the backward form as the following two theorems.
\begin{thm}
  Let $u$ be a weak solution of a quasi-linear parabolic equation
  \begin{equation}\label{1.2}
  \begin{array}{l}
  \begin{split}
    -\partial_{t}u=\partial_{x_i} \mathscr{A}_i(t,x,u,\nabla u)+\mathscr{B}(t,x,u,\nabla u)
    \end{split}
  \end{array}
 \end{equation}
  in the bounded cylinder $Q=(0,T)\times\cO\subset \bR^{1+n}$ such that $u\leq M$ on the parabolic boundary $\left((0,T]\times \cO\right)\cup\left( \{T\}\times \cO \right)$.
  Then almost everywhere in $Q$
  $$u\leq M +C \Xi(\mathscr{A},\mathscr{B}) $$
  where the constant $C$ depends only on $T,|\cO|$ and the structure terms of the equation, while $\Xi(\mathscr{A},\mathscr{B})$ is expressed in terms of some quantities related to the coefficients $\mathscr{A}$ and $\mathscr{B}$.
\end{thm}
\begin{thm}
  Let $u$ be a weak solution of ~(\ref{1.2}) in $Q$. Suppose that the set $Q_{3\rho}$ is contained in $Q$. Then almost everywhere in $Q_{\rho}$ we have
  $$|u(t,x)|\leq C \left(\rho^{-(n+2)/2}\|u\|_{W^{2}(Q(3\rho))}+\rho^{\theta} \Xi_1(\mathscr{A},\mathscr{B})    \right)$$
  where the constant $C$ depends only on $\rho$ and the structure terms of   ~(\ref{1.2}), $Q_{\rho}:=(\bar{t},\bar{t}+\rho^2) \times B_{\rho}(\bar{x})$, $\theta\in (0,1)$ is one of the structure terms of   ~(\ref{1.2}) and $\Xi_1(\mathscr{A},\mathscr{B})$ is expressed in terms of some quantities related to the coefficients $\mathscr{A}$ and $\mathscr{B}$.
  In particular, weak solutions of   (\ref{1.2}) must be locally essentially bounded.
\end{thm}
In contrast with the deterministic one, the stochastic maximum principle has received rather few discussions. We note that Denis and
Matoussi~\cite{DenisMatoussi2009}, and Denis, Matoussi, and Stoica~\cite{DenisMatoussiStoica2005} gave a stochastic version of Aronson and
Serrin's above results, and obtained  via Moser's iteration scheme a stochastic maximum principle, which claims an $L^p$ estimate for the time
and space maximal norm of weak solutions to {\it forward} quasi-linear stochastic partial differential equations (SPDEs). Any stochastic
maximum principle  seems to be lacking for {\it backward} ones in the literature, which then becomes quite interesting to know.

In this paper, we concern the maximum principle of a weak solution to BSPDE~(\ref{1.1}). Using the De Giorgi iteration scheme, we establish
the global maximum principle and the local boundedness theorem for quasi-linear BSPDEs~(\ref{1.1}), which include the above two theorems as
particular cases. As highlighted by the classical theory of deterministic parabolic PDEs, our stochastic maximum principle for BSPDEs is
expected to be used in the study of H\"older continuity of the solutions of BSPDEs and further in the study of  more general quasi-linear
BSPDEs.

 It is worth noting that our estimates for weak  solutions are uniform with respect to
$w\in \Omega$. In contrast to Denis, Matoussi, and Stoica's  $L^p$ estimate ($p\in (2,\infty)$) for the time and space maximal norm of weak
solutions of ({\it forward}) quasi-linear SPDEs, we prove an $L^{\infty}$ estimate for that of quasi-linear BSPDE (\ref{1.1}). This
distinction comes from the essential difference between SPDEs and BSPDEs: the diffusion $v$ in BSPDE (\ref{1.1}) is endogenous, while the
diffusion in the SPDEs is exogenous, which makes impossible any $L^\infty$ estimate  for a forward SPDE due to the active white noise. On the
other hand, indeed, the technique of Moser's iteration can also be used to study the behavior of weak solutions of BSPDE (\ref{1.1}) and to
obtain the global and local maximum principles. However, as the De Giorgi iteration scheme works for the degenerate parabolic case, we prefer
De Giorgi's method in this paper and leave the application of Moser's method as an exercise to the interested reader.

Many works have been devoted to  the linear and semi-linear BSPDEs either in the whole space or in a domain (see, for instance,
\cite{Dokuchaev2010,DuQiuTang10,DuTang2009,Hu_Ma_Yong02,Tang_05,Zhou_92,Zhou_93}). A theory of solvability of quausi-linear BSPDEs is recently
established in an abstract framework in   Qiu and Tang \cite{QiuTangBDSDES2010}. However, it is prevailing in these works to assume that the
coefficients $b,c$ and $\varsigma$ are essentially bounded. To inherit in our stochastic maximum principle the general structure  of admitting
the unbounded coefficients $b$ and $c$ in the deterministic maximum principle,
% However, previous works on the  maximum principles seem to be restricted within  descriptive comparison analysis for the linear or semi-linear BSPDEs without any maximal estimates.
%To the authors' best knowledge,  neither global nor local maximum principle is addressed for our quasi-linear BSPDE (\ref{1.1}) in the literature.
we prove by approximation in Section 4 the existence and uniqueness result (Theorem 4.1) for the weak solution to the quasi-linear BSPDE
(\ref{1.1}) with the null Dirichlet condition on the lateral boundary, under a new rather general framework. This result is invoked to prove
Proposition 4.3 as the It\^o's formula for the composition of solutions of BSDEs into a class of time-space smooth functions, which is the
starting point of the De Giorgi scheme in the proof of subsequent stochastic maximum principles.

This paper is organized as follows. In Section 2, we set notations, hypotheses and  the notion of the weak solution to BSPDE (\ref{1.1}). In
Section 3, we prepare several auxiliary results, including a generalized It\^o formula,  which will be used to establish Proposition 4.3 below
as a key step in the  study of our stochastic maximum principle. In Section 4 we prove the existence and uniqueness of the weak solution to
BSPDE (\ref{1.1}). Finally, in Section 5, we establish the maximum principles for quasi-linear BSPDEs. In the first subsection,   we use the
De Giorgi iteration scheme to obtain the global maximum principles for BSPDEs (\ref{1.1}) and in the second subsection, we prove the local
maximum principle for our backward stochastic parabolic De Giorgi class.

\section{Preliminaries}

Let $(\Omega,\sF,\{\sF_t\}_{t\geq0},\bP)$ be a complete filtered probability space on which is defined an $m$-dimensional standard Brownian
motion $W=\{W_t:t\in[0,\infty)\}$ such that $\{\sF_t\}_{t\geq0}$ is the natural filtration generated by $W$ and augmented by all the
$\bP$-null sets in $\sF$. We denote by $\sP$ the $\sigma$-algebra of the predictable sets on $\Omega\times[0,T]$ associated with
$\{\sF_t\}_{t\geq0}$.

Denote by $\mathbb{Z}$ the set of all the integers and by $\bN$ the set of all the positive integers. Denote by $|\cdot|$ and $\langle
\cdot,\cdot\rangle$ the norm and scalar product in a finite-dimension Hilbert space. Like in  $\bR,\bR^k,\bR^{k\times l}$ with $k,l\in \bN$,
we have defined
$$|x|:=\left( \sum_{i=1}^k x^2_i  \right)^{\frac{1}{2}} \quad
\textrm{and}\quad  |y|:=\left(\sum_{i=1}^k\sum_{j=1}^l y^2_{ij} \right)^{\frac{1}{2}} \quad \textrm{for}~ (x,y)\in \bR^k \times \bR^{k\times
l}.
$$
For the sake of convenience, we denote
$$\partial_{s}:=\frac{\partial} {\partial {s}} \ \, {\rm and } \ \, \partial_{st}:=\frac{\partial^2}{\partial s\partial t}.$$

Let $V$ be a Banach space equipped with norm $\|\cdot\|_V$.
%For a $V$-valued, adapted and
% c$\grave{\textrm{a}}$dl$\grave{\textrm{a}}$g process $X=(X_{t})_{t\in [0,T]}$, we denote  $ \sup_{t\in [0,T]}\|X_t\|_V$ either by $X_{*}$ or by $\sup_t \|X_t\|_V$.
  For  real $p\in
 (0,\infty)$, $\cS ^p (V)$ is the set of all the $V$-valued,
 adapted and c$\grave{\textrm{a}}$dl$\grave{\textrm{a}}$g processes $(X_{t})_{t\in [0,T]}$ such
 that
 $$\|X\|_{\cS ^p(V)}:= \left(E [\sup_{t\in [0,T]} \|X_t\|_V^p] \right)^{1\wedge\frac{1}{p}}< \infty.$$
It is worth noting that ($\cS ^p (V)$, $\|\cdot\|_{\cS^p(V)}$) is a Banach space for $p\in [1,\infty)$ and for $p\in (0,1)$, $dis(X,X'):= \|X-X'\|_{\cS^p(V)}$ is a metric of $\cS^p(V)$ under which $\cS^p(V)$ is complete.

 Define the parabolic distance in $\bR^{1+n}$ as follows:
$$\delta(X,Y):=\max\{ |t-s|^{1/2},|x-y| \},$$
 for $X:=(t,x)$ and $Y:=(s,y)\in \bR^{1+n}$. Denote by $Q_r(X)$ the ball of radius $r>0$ and center $X:=(t,x)\in \bR^{1+n}$ with $x\in \bR^n$:
\begin{equation*}
  \begin{split}
    Q_r(X):=&\, \{ Y\in \bR^{1+d}:\delta(X,Y)<r \} = (t-r^2,t+r^2)\times B_r(x),\\
    B_r(x):=&\, \{ y\in\bR^n:|y-x|<r \},
  \end{split}
\end{equation*}
and by $|Q_r(X)|$ the volume.

Denote by $\partial\Pi$ the boundary of domain $\Pi\subset\bR^n$. Throughout this paper, we assume $\partial \cO\in C^1$. The set
$S_T:=[0,T]\times\partial\cO$ is called the lateral boundary of $Q$ and the set $\partial_{\rm p} Q:=S_T\cup(\{T\}\times\cO)$ is called the
parabolic boundary of $Q$.

For  domain $\Pi\subset\bR^n$, we denote by $C_c^{\infty}(\Pi)$ the totality of infinitely differentiable functions of compact supports in
$\Pi$, and the spaces like $L^{\infty}(\Pi),L^p(\Pi)$ and $W^{k,p}(\Pi)$ are defined as usual for integer $k$ and real number $p\in
[1,\infty)$. We denote by $\ll\cdot,~\cdot\gg_{\Pi}$ the inner product of $L^2(\Pi)$ and the subscript $\Pi$  will be omitted for $\Pi =\cO$.
Set $\Pi_t :=[t,T]\times \Pi$ for $t\in [0,T)$.
 For each integer $k$ and real number $p\in [1,\infty)$, we denote by $W^{k,p}_{\sF}(\Pi_t )$ the totality of the $W^{k,p}(\Pi)$-valued predictable processes $u$ on $[t,T]$ such that
 $$ \|u\|_{W^{k,p}_{\sF}(\Pi_t )}:=\left(E\left[ \int_t^T\|u(s,\cdot)\|_{W^{k,p}(\Pi)}^p ds\right]   \right)^{1/p}<\infty.$$
Then $(W^{k,p}_{\sF}(\Pi_t ),~\|\cdot\|_{W^{k,p}_{\sF}(\Pi_t )})$ is a Banach space.

\bigskip
\begin{defn}\label{def of basic space}
  For $(p,t,k)\in  [1,\infty)\times [0,T)\times \mathbb{Z}$, define $\cM^{k,p}(\Pi_t )$ as the totality of
  $u\in W^{k,p}_{\sF}(\Pi_t )$ such that
  $$ \|u\|_{k,p;\Pi_t }:=\left( \esssup_{\omega\in\Omega} \sup_{s\in [t,T]}E\left[\int_s^T\|u(\omega,\tau,\cdot)\|^p_{W^{k,p}(\Pi)}d\tau\big|\sF_s \right] \right)^{1/p}<\infty. $$
\end{defn}

\bigskip
  For $u\in W^{k,p}_{\sF}(\Pi_t )$, we deduce from~\cite[Theorem 6.3]{Hu_2002} that the process
  $$\left\{ 1_{[t,T]}(s)E\left[ \int_s^T\|u(\omega,\tau,\cdot)\|^p_{W^{k,p}(\Pi)}d\tau\big|\sF_s   \right], \,\, s\in [0,T]\right\} \quad
  \in S^{\beta}(\bR) \textrm{ for any } \beta\in(0,1).$$
  This shows that the norm $\|\cdot\|_{k,p;\Pi_t }$ in the preceding definition makes a sense.
 Moreover, ($\cM^{k,p}(\Pi_t )$, $\|\cdot\|_{k,p;\Pi_t }$) is a Banach space.

To simplify notations, $k=0$ appearing in either superscript or subscript of spaces or norms will be omitted and therefore the notations
$W^{0,p}_{\sF}(\Pi_t ),~ \|\cdot\|_{W^{0,p}_{\sF}(\Pi_t )},~\cM^{0,p}(\Pi_t ) $ and $\|\cdot\|_{0,p;\Pi_t }$ will be abbreviated as
$W^p_{\sF}(\Pi_t ),~ \|\cdot\|_{W^p_{\sF}(\Pi_t )},~\cM^p(\Pi_t ) $ and $\|\cdot\|_{p;\Pi_t }$.  Note that $W^{0,p}(\Pi)\equiv L^p(\Pi)$.

Moreover, we introduce the following spaces of random fields. $\cL^{\infty}(\Pi_t )$ is the totality of $u\in W^p_{\sF}(\Pi_t )$ such that
$$\|u\|_{\infty;\Pi_t }:= \esssup_{(\omega,s,x)\in \Omega\times[t,T]\times\Pi}|u(\omega,s,x)|<\infty.  $$
$\cL^{\infty,p}(\Pi_t )$ is the totality of $u\in W^p_{\sF}(\Pi_t )$ such that
$$\|u\|_{\infty,p;\Pi_t }:= \esssup_{(\omega,s)\in \Omega\times[t,T]}\|u(\omega,s,\cdot)\|_{L^p(\Pi)}<\infty. $$
 $\cV_2(\Pi_t )$ is the totality of $u\in W^{1,2}_{\sF}(\Pi_t )$ such that
\begin{equation}\label{norm of space V_2}
\|u\|_{\cV_2(\Pi_t )}:=\left(\|u\|^2_{\infty,2;\Pi_t }+\|\nabla u\|^2_{2;\Pi_t }  \right)^{1/2}<\infty.
\end{equation}
$\cV_{2,0}(\Pi_t )$, equipped with the norm \eqref{norm of space V_2}, is the totality of $u\in \cV_2(\Pi_t )$ such that
$$ \lim_{r\rightarrow 0}\|u(s+r,\cdot)-u(s,\cdot)\|_{L^2(\Pi)}=0 ,\textrm{ for all }s,s+r\in [t,T]$$
holds almost surely. We denote by $\dot{\cV}_2(Q)$ ($\dot{\cV}_{2,0}(\Pi_t )$, $\dot{W}^{1,p}_{\sF}(\Pi_t )$ and $\dot{\cM}^{1,p}(\Pi_t )$, respectively) all the random fields $u\in\cV_2(Q)$ ($\cV_{2,0}(\Pi_t )$, $W^{1,p}_{\sF}(\Pi_t )$ and $\cM^{1,p}(\Pi_t )$, respectively),
 satisfying
           $$u(\omega,s,\cdot)|_{\partial \Pi}=0,\quad a.e.~ (\omega,s)\in\Omega\times [t,T].$$

By convention, we treat elements of spaces defined above like $W^{k,p}(\Pi)$ and $\cM^{k,p}(\Pi_t )$
 as functions rather than distributions or classes of equivalent functions,
and if we know that a function of this class has a modification with better properties, then we always consider this modification. For
example, if $u\in W^{1,p}(\Pi)$ with $p>n$, then $u$ has a modification lying in $C^{\alpha}(\Pi)$ for $\alpha\in (0,\frac{p-n}{p})$, and we
always adopt the modification $u\in W^{1,p}(\Pi)\cap C^{\alpha}(\Pi)$. By saying a finite dimensional vector-valued function $v:=(v_{i})_{i\in
\mathcal {I}}$ belongs to a space like $W^{k,p}(\Pi)$, we mean that each component $v_i$ belongs to the space and the norm is defined by
$$\|v\|_{W^{k,p}(\Pi)}=\left( \sum_{i\in\mathcal{I}}\|v_{i}\|^p_{W^{k,p}(\Pi)}  \right)^{1/p}.$$

\medskip
Consider quasi-linear BSPDE~(\ref{1.1}). We define the following assumptions.

\bigskip\medskip
   $({\mathcal A} 1)$ \it The pair of  random functions
\begin{equation*}
  f(\cdot,\cdot,\cdot,\vartheta,y,z):~\Omega\times[0,T]\times\cO\rightarrow\bR^n
  \textrm{ and }
  g(\cdot,\cdot,\cdot,\vartheta,y,z):~\Omega\times[0,T]\times\cO\rightarrow\bR
\end{equation*}
are $\sP\otimes\cB(\cO)$-measurable for any $(\vartheta,y,z)\in \bR\times\bR^{n}\times\bR^{ m}$. There exist positive constants $L,\kappa$
and $\beta$
   such that for all $(\vartheta_1,y_1,z_1),(\vartheta_2,y_2,z_2)\in \bR\times\bR^n\times\bR^{n\times m}$
   and $(\omega,t,x)\in \Omega\times[0,T]\times\cO$
   \begin{equation*}
     \begin{split}
       |f(\omega,t,x,\vartheta_1,y_1,z_1)-f(\omega,t,x,\vartheta_2,y_2,z_2)|\leq& L|\vartheta_1-\vartheta_2|+\frac{\kappa}{2}|y_1-y_2|+\beta^{1/2}|z_1-z_2|,\\
       |g(\omega,t,x,\vartheta_1,y_1,z_1)-g(\omega,t,x,\vartheta_2,y_2,z_2)|\leq& L(|\vartheta_1-\vartheta_2|+|y_1-y_2|+|z_1-z_2|).
     \end{split}
   \end{equation*}\rm

\medskip
   $({\mathcal A}2)$ \it The pair functions $a$ and $\sigma$ are $\sP\otimes\cB(\cO)$-measurable. There exist positive constants $\varrho>1, \lambda$ and $\Lambda$ such that
   the following  hold for all $\xi\in\bR^n$ and $(\omega,t,x)\in \Omega\times[0,T]\times\cO$
   \begin{equation*}
     \begin{split}
       &\lambda|\xi|^2\leq (2a^{ij}(\omega,t,x)-\varrho\sigma^{ir}\sigma^{jr}(\omega,t,x))\xi^i\xi^j\leq \Lambda|\xi|^2;\\
       &|a(\omega,t,x)|+|\sigma(\omega,t,x)|\leq \Lambda;\\
       &\hbox{ \rm and  }\lambda-\kappa-\varrho'\beta>0 \textrm{ \rm with }\varrho':=\frac{\varrho}{\varrho-1}.
     \end{split}
   \end{equation*}\rm

\medskip
   $({\mathcal A} 3)$ \it $G\in L^{\infty}(\Omega,\sF_T,L^2(\cO))$. There exist two real numbers $p>n+2$ and $q>(n+2)/2$ such that
$$f_0:=f(\cdot,\cdot,\cdot,0,0,0)\in\cM^p(Q),\ g_0:=g(\cdot,\cdot,\cdot,0,0,0)\in\cM^{\frac{p(n+2)}{p+n+2}}(Q),$$
and $ \left(b^i\right)^2, \left(\varsigma^r\right)^2, c\in \cM^q(Q)$, $i=1,\cdots,n$; $r=1,\cdots,m$.
     Define
   \begin{equation}
\Lambda_0:=B_q(b,c,\varsigma):=\||b|^2\|_{q;Q}+\|c\|_{q;Q}+\||\varsigma|^2\|_{q;Q}.
\end{equation}\rm

\medskip
   $({\mathcal A} 3)_0$ \it $\quad G\in L^{\infty}(\Omega,\sF_T,L^2(\cO)),\,f_0\in\cM^2(Q),\,g_0\in\cM^2(Q)$
    and   $b,\,\varsigma,\,c\in \cL^{\infty}(Q)$. \rm

\medskip
 $({\mathcal A} 4)$  \it There exists a nonnegative constant $L_0$ such that
    $c\leq L_0$.\rm

\medskip\bigskip
%Define $$f_0:=f(\cdot,\cdot,\cdot,0,0,0), \quad g_0:=g(\cdot,\cdot,\cdot,0,0,0).$$
For $p\in [2,\infty)$, define the functional $A_p$:
$$
A_p(u,v):=\|u\|_{p;Q}+\|v\|_{{\frac{p(n+2)}{p+n+2}};Q},\quad (u,v)\in \cM^p(Q)\times\cM^{\frac{p(n+2)}{p+n+2}}(Q),
$$
and the functional $H_p$:
$$
H_p(u,v):=\|u\|_{p;Q}+\|v\|_{p;Q},\quad (u,v)\in \cM^p(Q)\times\cM^{p}(Q).
$$

\begin{defn}\label{definition of weak solution}
  A pair of processes $(u,v)\in W_{\sF}^{1,2}(Q)\times W_{\sF}^2(Q)$
  is called a weak solution to BSPDE (\ref{1.1})
  if it holds in the weak sense, i.e.
   for any $\varphi\in C_c^{\infty}(\cO)$ there holds almost surely
  \begin{equation}\label{eq in defn of weak sol}
    \begin{split}
      &\ll \varphi,\, u(t)\gg                               \\
      =&\ll \varphi,\,G\gg
        -\int_t^T\ll \varphi,\,v^r(s)\gg dW_s^r
      +\int_t^T\ll \varphi,\, g(s,\cdot,u(s),\nabla u(s),v(s))\gg  ds \\
      &-\int_t^T \ll \partial_{x_j} \varphi,
      \quad a^{ij}\partial_{x_i} u(s)+\sigma^{j r}v^r(s)
      +f^j(s,\cdot,u(s),\nabla u(s),v(s))\gg ds\\
      &+\int_t^T\ll
      \varphi,\,
            b^i\partial_{x_i}u(s)+c\,u(s)+\varsigma^rv^r(s)
            \gg ds,\quad\forall\, t\in [0,T].\\
    \end{split}
  \end{equation}

  Denote by $\sU\times\sV(G,f,g)$ the set of all the weak  solutions $(u,v)\in\cV_{2,0}(Q)\times \cM^2(Q)$
   of BSPDE (\ref{1.1}).
\end{defn}

\begin{rmk}\label{rmk defn weak solution}
Let $(u,v)\in W_{\sF}^{1,2}(Q)\times W_{\sF}^2(Q)$
  be a weak solution to BSPDE (\ref{1.1}).
For each $\zeta(t,x)=\psi(t)\varphi(x)$ with $\varphi\in C_c^{\infty}(\cO)$ and $\psi\in C_c^{\infty}(\bR)$, in view of   \eqref{eq in defn
of weak sol}, we have almost surely
\begin{equation*}
  \begin{split}
    &\ll \zeta(s''),\,u(s'')\gg - \ll \zeta(s'),\,u(s')\gg \\
    =&
    \ll \zeta(s'')-\zeta(s'),\, u(s'')\gg +\ll \zeta(s'),\, u(s'')-u(s')\gg\\
    =&
    [\psi(s'')-\psi(s')]\ll \varphi,\, u(s'')\gg
    +\psi(s')\left( \ll \varphi,\, u(s'')\gg -\ll \varphi,\, u(s')\gg \right)\\
    =&[\psi(s'')-\psi(s')]\ll \varphi ,\,u(s'')\gg\\
    &-\psi(s')\bigg(
    \int_{s'}^{s''}\ll \varphi,\, g(s,\cdot,u(s),\nabla u(s),v(s))\gg ds
     -\int_{s'}^{s''}\ll \varphi ,\, v^r(s)\gg dW_s^r \\
      &-\int_{s'}^{s''}\ll  \partial_{x_j} \varphi,\,  a^{ij}\partial_{x_i} u(s)
      +\sigma^{j r}v^r(s)+f^j(s,\cdot,u(s),\nabla u(s),v(s))\gg ds\\
      &+\int_{s'}^{s''}\ll
      \varphi,\,
            b^i\partial_{x_i}u(s)+c\,u(s)+\varsigma^rv^r(s)
            \gg ds
    \bigg)
  \end{split}
\end{equation*}
for $s''=t_{i+1}$ and $s'=t_i$, where $t=t_0<t_1<t_2<\cdots<t_N=T,\,\, 2\!<N\in\bN$ and $t_{i+1}-t_{i}=T/N$, $i=1,2,\cdots,N$. Summing up both sides
of these equations and passing to the limit, we have almost surely
  \begin{equation}\label{eq in rmk defn weak solution}
    \begin{split}
      &\ll \zeta(t),\,u(t)\gg                               \\
      =&\ll \zeta(T),\, G\gg -\int_t^T\ll \partial_s \zeta(s),\, u(s) \gg ds
      -\int_t^T \ll \zeta(s),\, v^r(s)\gg  dW_s^r\\
      &-\int_t^T \ll \partial_{x_j} \zeta(s),\quad a^{ij}\partial_{x_i} u(s)+\sigma^{j r}v^r(s)
      +f^j(s,\cdot,u(s),\nabla u(s),v(s)) \gg ds\\
    &+\int_t^T\ll
      \zeta(s),\,
            b^i\partial_{x_i}u(s)+c\,u(s)+\varsigma^rv^r(s)
            \gg ds     \\
     &  +\int_t^T \ll \zeta(s),\, g(s,\cdot,u(s),\nabla u(s),v(s))\gg ds, \quad \forall \ t\in[0,T].
    \end{split}
  \end{equation}
  Since the linear space
  $$\left\{\sum_{i=1}^N\psi_i(t)\varphi_i(x), (t,x)\in \bR\times \cO:    N\in \bN, \,
  (\varphi_i, \psi_i) \in C_c^{\infty}(\cO)\times C_c^{\infty}(\bR), \, i=1,2,\cdots,N\right\}$$
  is dense in $C_c^{\infty}(\bR)\otimes C_c^{\infty}(\cO)$, \eqref{eq in rmk defn weak solution} holds for any test function
  $\zeta\in C_c^{\infty}(\bR)\otimes C_c^{\infty}(\cO)$.
 \end{rmk}

 Under assumptions $({\mathcal A}1), ({\mathcal A}2)$ and $({\mathcal A}3)_0$, we deduce from \cite[Theorem 2.1]{QiuTangBDSDES2010} that there
exists a unique weak solution $(u,v)\in (\dot{W}^{1,2}_{\sF}(Q)\cap S^2(L^2(\cO)))\times W^2_{\sF}(Q)$,  which admits $L^2(\cO)$-valued
continuous trajectories for $u$, and which is also said to satisfy the null Dirichlet condition on the lateral boundary since $u$ vanishes in a
generalized sense on the boundary $\partial \cO$. Denote by $\dot{\sU}\times\dot{\sV}(G,f,g)$ all the random fields lying in
$\sU\times\sV(G,f,g)$ which satisfy the null Dirichlet boundary condition.

\section{Auxiliary results}

In what follows, $C>0$ is a constant which may vary from line to line and $C(a_1,a_2,\cdots)$ is a constant to depend on the parameters  $a_1,a_2,\cdots$.

First, we give the following embedding lemma.
\begin{lem}\label{lem emmbedding for space V}
   For $u\in \dot{\cV}_{2}(\Pi_t )$ with $t\in [0,T)$, we have $u\in \cM^{\frac{2(n+2)}{n}}(\Pi_t )$ and
   \begin{equation*}
     \begin{split}
       \|u\|_{\frac{2(n+2)}{n};\Pi_t }\leq\,&
         \ C(n) ~\|\nabla u\|_{2;\Pi_t }^{n/(n+2)}
          \esssup_{(\omega,s)\in\Omega\times[t,T]}\|u(\omega,s,\cdot)\|_{L^2(\Pi)}^{2/(n+2)}
         \leq\  C(n)~ \|u\|_{\cV_2(\Pi_t )}.
     \end{split}
   \end{equation*}
\end{lem}
\begin{proof}
  By the well known Gagliard-Nirenberg inequality (c.f. \cite{FriedmanPDEs}, \cite{Ladyzhenskaia_68} or \cite{Nirenberg1959}), we have
  \begin{equation*}
    \|u(\omega,s,\cdot)\|_{L^q(\Pi)}^q
    \leq\  C~
        \|\nabla u(\omega,s,\cdot)\|_{L^2(\Pi)}^{\alpha q}\|u(\omega,s,\cdot)\|_{L^2(\Pi)}^{q(1-\alpha)},
        \quad a.e.\ (\omega,s)\in\Omega\times[t,T],
  \end{equation*}
  where $\alpha=n/(n+2)$ and $q=2(n+2)/n$.
  Integrating on $[\tau,T]$ for $\tau\in [t,T)$ and taking conditional expectation, we obtain almost surely
  \begin{equation*}
    \begin{split}
    E\left[\int_{\Pi_{\tau} }|u(s,x)|^qdxds\Big|\sF_{\tau} \right]
    \leq &\ C~
    \|\nabla u\|_{2;\Pi_t }^{2}
        \esssup_{(\omega,s)\in\Omega\times[t,T]}\|u(\omega,s,\cdot)\|_{L^2(\Pi)}^{(1-\alpha) q}
    \leq \ C~
        \|u\|^q_{\cV_2(\Pi_t )}.
    \end{split}
  \end{equation*}
Therefore,  $u\in \cM^{\frac{2(n+2)}{n}}(\Pi_t )$ and
   \begin{equation*}
     \begin{split}
       \|u\|_{\frac{2(n+2)}{n};\Pi_t }\leq&
         \ C ~\|\nabla u\|_{2;\Pi_t }^{n/(n+2)}
          \esssup_{(\omega,s)\in\Omega\times[t,T]}\|u(\omega,s,\cdot)\|_{L^2(\Pi)}^{2/(n+2)}
         \leq\  C~ \|u\|_{\cV_2(\Pi_t )}
     \end{split}
   \end{equation*}
   with $C$ only depending on $n$.
\end{proof}

\begin{lem}
  For any $r\in \bR$ and $u\in\cV_{2,0}(\Pi_t )$ with $t\in [0,T)$ we have
    $$(u-r)^+:=(u-r)\vee 0\in \cV_{2,0}(\Pi_t ).$$
    Moreover, if $\{u_k,k\in\bN\}$ is a Cauchy sequence in $\cV_{2,0}(\Pi_t )$ with limit $u\in \cV_{2,0}(\Pi_t )$,
    then
    $$\lim_{k\to \infty}\|(u_k-r)^+-(u-r)^+\|_{\cV_2(\Pi_t )}=0.$$
\end{lem}
\begin{proof}
   It can be checked that $(u-r)^+\in\cV_{2}(\Pi_t )$. Since
   $$|(u-r)^+-(v-r)^+|\leq |u-v|,$$
   Then we have
   $$\|(u-r)^+(s+h)-(u-r)^+(s)\|_{L^2(\Pi)}\leq \|u(s+h)-u(s)\|_{L^2(\Pi)}, \quad \forall s,\, s+h\in [t,T].$$
   Hence, the continuity of $u$ implies that of $(u-r)^+$. The other assertions follow in a similar way. We complete our proof.
\end{proof}
%The proof of the lemma is similar to the deterministic case (c.f. \cite{ChenyazheParabolic}). It can be deduced from the fact that $$|(u-r)^+-(v-r)^+|\leq |u-v|. $$
%We omit the proof here.

  In contrast to the deterministic case,  the integrand of It\^o's stochastic integral is required to be adapted,
  and the technique of Steklov time average (see~\cite[page 100]{Lieberman}) finds difficulty in our stochastic situation.
  We directly establish some It\^o formula to get around the difficulty.

\begin{lem}\label{lem ito foumul}
Let $\phi:\bR\times\bR^n\times\bR\longrightarrow\bR$ be a continuous function
which is twice continuously differentiable such that
%of continuous derivatives $\partial_{x_i} \phi(t,x,s)$, $\partial_{t}\phi(t,x,s)$,
% $\phi'(t,x,s):=\partial_{s}\phi(t,x,s)$,
% $\phi''(t,x,s):=\partial_{ss}\phi(t,x,s)$ and
% $\partial_{x_i} \phi'(t,x,s),~i=1,2,\cdots,n$ on
%  $\bR\times\bR^n\times\bR$.
   $\phi'(t,x,0)=0$
 for any $(t,x)\in \bR\times\bR^n$ and
 there exists a constant $M\in (0,\infty)$ such that
 $$\sup_{(t,x)\in\bR^{n+1},s\in\bR\setminus \{0\}}
 \left\{\left|\phi''(t,x,s)\right|
  +\frac{1}{|s|}\sum_{i=1}^n \left|\partial_{x_i} \phi'(t,x,s)
   \right|
  +\frac{1}{s^2}\left| \partial_{t} \phi(t,x,s)
  -\partial_{t} \phi(t,x,0)  \right| \right\}<M,$$
  where $\phi'(t,x,s):=\partial_{s}\phi(t,x,s)$ and
 $\phi''(t,x,s):=\partial_{ss}\phi(t,x,s)$.
Assume that the equation
  \begin{equation}\label{eq bspde in ito formula}\begin{split}
    u(t,x)=\, &u(T,x)+\int_t^T\left(h^0(s,x)+\partial_{x_i}  h^i(s,x)\right)\, ds-\int_t^Tz^r(s,x)\, dW_s^r, \quad t\in[0,T]
    \end{split}
  \end{equation}
   holds in the weak sense of Definition \ref{definition of weak solution},
where $u(T)\in L^2(\Omega,\sF_T,L^2(\cO));\  h^i\in W^{2}_{\sF}(Q), i=0,1,\cdots,n;$ and $z\in W^{2}_{\sF}(Q)$. If $u\in
\dot{W}^{1,2}_{\sF}(Q)\cap S^2(L^2(\cO))$,
 we have almost surely
 \begin{equation}\label{eq ito lemma}
   \begin{split}
     &\int_{\cO}\phi(t,x,u(t,x))\, dx \\
     =&\int_{\cO}\phi(T,x,u(T,x))\, dx
     -\int_t^T\int_{\cO}\partial_{s}\phi(s,x,u(s,x))\,dxds \\
          &-\int_t^T\ll\phi'(s,\cdot,u(s)),\,z^r(s)\gg dW_s^r
     +\int_t^T \ll \phi'(s,\cdot,u(s)) ,\, h^0(s) \gg ds\\
     &-\int_t^T \ll \phi''(s,\cdot,u(s))\partial_{x_i} u(s)+
     \partial_{x_i} \phi'(s,\cdot,u(s)),\, h^i(s)\gg  ds              \\
     &-\frac{1}{2}\int_t^T\ll \phi''(s,\cdot,u(s)),\, |z(s)|^2\gg ds,~\forall t\in [0,T].
   \end{split}
 \end{equation}
% holds for all $t\in [0,T]$, a.s..
\end{lem}

\begin{rmk}
    Lemma \ref{lem ito foumul} extends  It\^o formulas of \cite{DenisMatoussiStoica2005}
    and \cite{DaPrato1992} to our more general case where the test function $\phi$ is allowed to depend on both time and space variables.
    The extension is motivated by the subsequent study of the local maximum principle
    where It\^o formula for truncated solutions of BSDEs is required.
\end{rmk}

\begin{proof}[{\bf Proof of Lemma \ref{lem ito foumul}}]
  All the integrals in \eqref{eq ito lemma} are well defined.
   In particular, the stochastic integral
   $$I(t):=\int_0^t\ll \phi'(s,\cdot,u(s)),\, z^r(s)\gg dW_s^r, \quad t\in [0,T]$$ is a martingale
   since
  \begin{equation*}
    \begin{split}
      E\left[\sup_{t\in [0,T]}\left|I(t)\right|\right]\le
     &\, CE\left[\left(\int_0^T\Bigm |\ll |\phi'(s,\cdot,u(s))|,\,|z(s)|\gg \Bigm|^2ds\right)^{1/2}\right]\\
     \leq &
     \, CM \|u\|_{S^2(L^2(\cO))}\|z\|_{W^2_{\sF}(Q)}.
    \end{split}
  \end{equation*}
    We extend the random fields $u,h^0,h^1,\cdots,h^n$ and $z$ from their domain $\Omega\times [0,T]\times \cO$ to  $\Omega\times [0,T]\times \bR^n$
    by setting them all to be zero outside $\cO$, and we still use themselves to denote their respective extensions. Since $u$ satisfies the
    null Dirichlet condition on the lateral boundary and $\partial\cO\in C^1$, we have $u\in W^{1,2}_{\sF}([0,T]\times \bR^n)$. It is obvious that all the
    extensions $h^0,h^1,\cdots,h^n$ and $z$ lie in $W^2_{\sF}([0,T]\times \bR^n)$.

  \textbf{Step 1.}
  Consider $h^i\in \dot{W}^{1,2}_{\sF}(\cO)$, $i=1,2,\cdots,n$.
  Choose a sufficiently large positive integer $N_0$  so that $\{x\in\cO:dis(x,\partial\cO)>1/N_0\}$ is a nonempty sub-domain of $\cO$.
  For integer $N>N_0$, define
  $$\cO^N:=\{x\in\cO:dis(x,\partial\cO)>1/N\}.$$
  Let $\rho\in C_c^{\infty}(\bR^n)$ be a nonnegative function such that
  $$\textrm{supp} ({\rho})\subset B_1(0) \hbox{ \rm and } \int_{\bR^n}\rho(x)\, dx=1.$$
   Define for each positive integer $k$,
   $$\rho_k(x):=(2Nk)^n\rho(2Nkx),\quad u_k(s,x):=u(s)\ast \rho_k(x):=\int_{\bR^n}\rho_k(x-y)u(s,y)\,dy.$$
  In a similar way, we write
  $$z_k(s,x):=z(s)\ast\rho_k(x) \hbox{ \rm  and } h^i_k(s,x):=h^i(s)\ast\rho_k(x), \, \, i=1,2,\cdots,n.$$
  Then for each $x\in \cO^N$, we have almost surely
  $$u_k(t,x)=u_k(T,x)+\int_t^T\left(\partial_{x_i} h_k^i(s,x)+h_k^0(s,x)\right)\,ds-\int_t^Tz_k^r(s,x)\,dW_s^r, ~\forall t\in [0,T].   $$
 By using It\^o formula for each $x\in\cO^N$  and then integrating over $\cO^N$ with respect to $x$ , we obtain
  \begin{equation}\label{eq 2 in ito formula}
   \begin{split}
     &\int_{\cO^N}\phi(t,x,u_k(t,x))\, dx \\ =&\int_{\cO^N}\!\phi(T,x,u_k(T,x))\, dx
     -\int_t^T\!\!\!\int_{\cO^N}\partial_{s}\phi(s,x,u_k(s,x))\, dxds\\
     &
     +\int_t^T\ll \phi'(s,\cdot,u_k(s)),\ \, h^0(s)\gg_{\cO^N} ds
     \\
     &+\int_t^T\ll \phi'(s,\cdot,u_k(s)),\ \, \partial_{x_i} h_k^i(s)\gg_{\cO^N} ds              \\
     &-\frac{1}{2}\int_t^T \ll \phi''(s,\cdot,u_k(s)),\ \, |z_k(s)|^2\gg_{\cO^N} ds\\
     &-\int_t^T\ll \phi'(s,\cdot,u_k(s)),\ \, z_k^r(s)\gg_{\cO^N} dW_s^r.
   \end{split}
 \end{equation}

For the sake of convenience, we define
\begin{equation*}
  \begin{split}
    \delta\phi_k(t,x):=&\phi(t,x,u(t,x))-\phi(t,x,u_k(t,x))\\
    \delta u_k(t,x):=& u(t,x)-u_k(t,x).
  \end{split}
\end{equation*}
%$$(\delta\phi_k,\delta u_k)(t,x)=(\phi(t,x,u(t,x))-\phi(t,x,u_k(t,x)),u(t,x)-u_k(t,x)).$$
and in a similar way, we define $\delta \phi_k',\delta\phi_k'',\delta h_k^i$ and $\delta z_k^r$ $i=0,1,\cdots,n;r=1,\cdots,m$.

  Since for almost all $(\omega,s)\in\Omega\times[0,T]$
 \begin{equation*}
   \begin{split}
     &\|u_k(\omega,s)\|_{W^{1,2}(\bR^n)}\leq\|u(\omega,s)\|_{W^{1,2}(\bR^n)},
     ~\lim_{k\rightarrow\infty}\|\delta u_k(\omega,s)\|_{W^{1,2}(\bR^n)}\rightarrow 0;\\
     &\|h^0_k(\omega,s)\|_{L^2(\bR^n)}\leq\|h^0(\omega,s)\|_{L^2(\bR^n)},
     ~\lim_{k\rightarrow\infty}\|\delta h^0_k(\omega,s)\|_{L^{2}(\bR^n)}\rightarrow 0;\\
     &\|h^i_k(\omega,s)\|_{W^{1,2}(\bR^n)}\leq\|h^i(\omega,s)\|_{W^{1,2}(\bR^n)},
     ~\lim_{k\rightarrow\infty}\|\delta h^i_k(\omega,s)\|_{W^{1,2}(\bR^n)}\rightarrow 0,i=1,2,\cdots;\\
     &\|z_k(\omega,s)\|_{L^2(\bR^n)}\leq\|z(\omega,s)\|_{L^2(\bR^n)},
     ~\lim_{k\rightarrow\infty}\|\delta z_k(\omega,s)\|_{L^{2}(\bR^n)}\rightarrow 0,
   \end{split}
 \end{equation*}
 by Lebesgue's dominated convergence theorem, we have as $k\rightarrow \infty$
 \begin{equation*}
   \begin{split}
      &\sum_{i=1}^n\|\delta h^i_k(s)\|^2_{W_{\sF}^{1,2}([0,T]\times \bR^n)}+\|\delta h^0_k(s)\|^2_{W^2_{\sF}([0,T]\times \bR^n)} +\|\delta z_k(s)\|_{W_{\sF}^2([0,T]\times \bR^n)}^2\\
      &+\|\delta u_k(s)\|_{W_{\sF}^{1,2}([0,T]\times \bR^n)}^2\rightarrow 0,
   \end{split}
 \end{equation*}
\begin{equation*}
  \begin{split}
        &E\left[\int_0^T\!\!\!\int_{\cO}\left|\delta\phi_k(t,x)\right|\, dxdt\right]
        \leq
        E\left[\int_0^T M\ll |u_k(t)|+|u(t)|,\, |\delta u_k(t)|\gg dt\right]\rightarrow 0,
  \end{split}
\end{equation*}
\begin{equation*}
  \begin{split}
    &E\left[\int_0^T\!\!\!\int_{\bR^n}\big|\phi'(s,x,u_k(s,x))\partial_{x_i} h_k^i(s,x)
    -\phi'(s,x,u(s,x))\partial_{x_i} h^i(s,x)\big|\, dxds\right]\\
    \leq&
        \ E\bigg[\int_0^T\!\!\!\int_{\bR^n}\Big(M\big|\delta u_k(s,x)\partial_{x_i} h^i_k(s,x)\big| +M|u(s,x)|\left|\partial_{x_i} (\delta h^i_k)(s,x)\right|\Big)\, dxds
        \bigg]
        \rightarrow 0,\\
        &\ i=1,\cdots,n
  \end{split}
\end{equation*}
and
\begin{equation*}
    \begin{split}
         &E\left[\int_0^T\!\!\!\int_{\bR^n}|\phi'(s,x,u_k(s,x))h_k^0(s,x)-\phi'(s,x,u(s,x))h^0(s,x)|\, dxds\right]\\
    \leq&
        E\left[\int_0^T\!\!\!\int_{\bR^n}\big(M|\delta u_k(s,x)h^0_k(s,x)|
        +M|u(s,x)||\delta h^0_k(s,x)|\big)\, dxds\right]
        \rightarrow 0.
    \end{split}
\end{equation*}

Since the convergence
$$\lim_{k\rightarrow \infty}\|\delta u_k\|_{W^{1,2}_{\sF}([0,T]\times\bR^n)}=0$$
implies that $u_k(\omega,t,x)$ converges to $u(\omega,t,x)$ in measure $dP\otimes dt \otimes dx$,
from the dominated
convergence theorem we conclude that
$$
    \lim_{k\rightarrow\infty}
     E\left[\int_0^T\!\!\!\int_{\cO}\,|{\partial_{s}}\phi(s,x,u_{k}(s,x))\,
    -\partial_{s}\phi(s,x,u(s,x))|\,
        \,dxds\right]=0.
$$
In a similar way,  we obtain
\begin{equation*}
  \begin{split}
    E\left[\int_0^T\!\!\!\int_{\cO}\Big|\phi''(s,x,u(s,x))|z(s,x)|^2
    -\phi''(s,x,u_k(s,x))|z_k(s,x)|^2\Big|\,dxds\right]\rightarrow 0
  \end{split}
\end{equation*}
and
 \begin{equation*}
   \begin{split}
        &E\left[\sup_{t\in[0,T]}\left|
        \sum_{r=1}^m
        \int_t^T\!\!\!\int_{\bR^n}
        \left(
        \phi'(s,x,u_k(s,x))z^r_k(s,x)
        -\phi'(s,x,u(s,x))z^r(s,x)\right)\,dxdW^r_s   \right|   \right]\\
     \leq&
        \ CE\left[\left(\int_0^T\left|\ll \phi'(s,\cdot,u_k(s)),\ z_k(s)\gg_{\bR^n}
        -\ll \phi'(s,\cdot,u(s)),\ z(s)\gg_{\bR^n} \right|^2\,ds\right)^{1/2} \right]\\
     \leq&
        \ CE\biggl[\Bigl(\int_0^T\big( \|\delta u_k(s)\|_{L^2(\bR^n)}^2 \|z(s)\|^2_{L^2(\bR^n)}
        + \|\phi'(s,u_k(s))\|^2_{L^2(\bR^n)} \|\delta z_k(s)\|^2_{L^2(\bR^n)} \big)\,ds\Bigr)^{1/2} \biggr]\\
        &\longrightarrow 0\textrm{ as }k\rightarrow \infty.
   \end{split}
 \end{equation*}
Hence taking limits in $L^1(\Omega\times[0,T],\sP)$ as $k\rightarrow \infty$ on both sides of \eqref{eq 2 in ito formula} and noting the path-wise continuity of $u$, we have almost surely
 \begin{equation}\label{eq N in ito lemma}
   \begin{split}
     &\int_{\cO^N}\phi(t,x,u(t,x))\, dx             \\
     =&\int_{\cO^N}\phi(T,x,u(T,x))\, dx
     -\int_t^T\!\!\!\int_{\cO^N} \partial_{s}\phi(s,x,u(s,x)) \,dxds \\
     &+
     \int_t^T\ll \phi'(s,\cdot,u(s)),\ \, h^0(s) \gg_{\cO^N} ds
     \\
     &+\int_t^T\ll\phi'(s,\cdot,u(s)),\ \,\partial_{x_i} h^i(s)\gg_{\cO^N} ds              \\
     &-\frac{1}{2}\int_t^T \ll \phi''(s,\cdot,u(s)),\ \,|z(s)|^2\gg_{\cO^N}ds\\
     &-\int_t^T\ll \phi'(s,\cdot,u(s)),\ \, z^r(s)\gg_{\cO^N} dW_s^r,\quad \forall\  t\in[0,T].
   \end{split}
 \end{equation}
  Passing to the limit in $L^1(\Omega\times[0,T],\sP)$ by letting $N\rightarrow \infty$ on both sides of \eqref{eq N in ito lemma}, in view of  the path-wise continuity of $u$ and the integration-by-parts formula, we conclude
  \eqref{eq ito lemma}.

  \textbf{Step 2.}
  For the general $h^i\in W_{\sF}^2(Q)$, we choose  sequences $\{h^i_k\}$, $\{z^r_k\}$ and $\{u_k\}$ from $S^2(\bR)\otimes C_c^{\infty}(\cO)$ such that
  \begin{equation*}
    \begin{split}
      \lim_{k\rightarrow\infty}\bigg\{
      &\sum_{i=0}^n \|\delta h_k^i\|_{W^{2}_{\sF}(Q)} +\|\delta z_k\|_{W_{\sF}^2(Q)}
          +\|\delta u_k\|_{W^{1,2}_{\sF}(Q)} +\|\delta u_k(0)\|_{L^2(\cO)}   \bigg\}=0.
    \end{split}
  \end{equation*}
  Consider
  \begin{equation}\label{eq step2 in ito thm}
    \begin{split}
      \bar{u}(t,x)=\, &u(0,x)+\int_0^t\left(\Delta \bar{u}(s,x)+\partial_{x_i} \tilde{h}^i(s,x)-h^0(s,x)\right)\, ds\\
      &+\int_0^t z^r(s,x)\, dW^r_s, \quad t\in [0,T]
    \end{split}
  \end{equation}
  with
  $$\tilde{h}^i(s,x):=-\partial_{x_i}  u(s,x)-h^i(s,x).$$
%   Then $u$ is a weak solution of \eqref{eq step2 in ito thm} in $\dot{W}^{1,2}_{\sF}(Q)\cap S^2(W^2(\cO))$.
    From Remark \ref{rmk defn weak solution} and \cite[Theorem 2.1]{Denis2004},  there are unique weak solutions $u\in \dot{W}^{1,2}_{\sF}(Q)\cap S^2(L^2(\cO))$
     to SPDE \eqref{eq step2 in ito thm}
   in the sense of \cite[Definition 1]{Denis2004} or equivalently \cite[Definition 4]{DenisMatoussiStoica2005}), and
    $u^k\in \dot{W}^{1,2}_{\sF}(Q)\cap S^2(L^2(\cO))$  to SPDE \eqref{eq step2 in ito thm}
    with $u(0,x)$, $z(s,x)$ and $\tilde{h}^i(s,x)$ being replaced by $u_k(0,x)$, $z_k(s,x)$ and
   $$\tilde{h}_k^i(s,x):=-\partial_{x_i}  u_k(s,x)-h_k^i(s,x), \quad k=1,2,\cdots.$$
   Then
   we deduce from \cite[Propositions 6 and 7, and Theorem 9]{Denis2004} that $u^k\in W^{2,2}_{\sF}(Q)\cap \dot{W}^{1,2}_{\sF}(Q)\cap S^2(L^2(\cO)) $ and
  \begin{equation}
  \begin{split}
  &\lim_{k\rightarrow\infty}\{\|u^k-u\|_{W_{\sF}^{1,2}(Q)}+\|u^k-u\|_{S^2(L^2(\cO))}\} \\
  \leq & \ C\lim_{k\rightarrow\infty}\{\|\delta u_k \|_{W_{\sF}^2(Q)}
  + \|\delta z_k\|_{W_{\sF}^2(Q)}
  +\|\delta u_k(0)\|_{L^2(\cO)}
  + \sum_{i=0}^n\|\delta h_k^i\|_{W_{\sF}^2(Q)}\}      \\
  =&\ 0
  \end{split}
  \end{equation}
  with the constant $C$ being independent of $k$.
  For each $k$, by \textbf{Step 1} we have
   \begin{equation*}
   \begin{split}
     &\int_{\cO}\phi(t,x,u^k(t,x))\, dx \\
     =&\int_{\cO}\phi(T,x,u^k(T,x))\, dx-\int_t^T\!\!\!\int_{\cO}\partial_{s}\phi(s,x,u^k(s,x))
     \,dxds\\
     &+\int_t^T\ll
     \phi'(s,\cdot,u^k(s)),\ \,h_k^0(s)\gg ds\\
     &+\int_t^T \ll
     \phi''(s,\cdot,u^k(s))\partial_{x_i} u^k(s)+
     \partial_{x_i} \phi'(s,\cdot,u^k(s)),\ \, \partial_{x_i} u^k(s)
                        \gg ds             \\
     &+\int_t^T\ll
     \phi''(s,\cdot,u^k(s))\partial_{x_i} u^k(s)+
     \partial_{x_i} \phi'(s,\cdot,u^k(s)),\ \,
     \tilde{h}_k^i(s)
                        \gg ds              \\
     &-\frac{1}{2}\int_t^T\ll
     \phi''(s,\cdot,u^k(s)),\, |z_k(s)|^2
                        \gg ds
     -\int_t^T\ll
     \phi'(s,\cdot,u^k(s)),\, z_k^r(s)
                        \gg dW_s^r,
   \end{split}
 \end{equation*}
  for all $t\in [0,T]$, $P$-a.s..
  By taking limits as $k\rightarrow \infty$,
  we complete our proof.
\end{proof}

\begin{rmk}\label{rmk ito formula 0}
  Let $\psi:\bR\times\bR^n\times\bR\longrightarrow\bR$ be a continuous function satisfying
  the assumptions on $\phi$ in Lemma \ref{lem ito foumul} except that for each $(t,y)$, $\psi''(t,y,s)$ may be not continuous with respect to $s$.
  Then if there exists a sequence $\{ \phi^k,k\in\bR \}$ of functions satisfying the  assumptions
  on $\phi$ in Lemma \ref{lem ito foumul}, such that
  $$\lim_{k\rightarrow \infty}\phi^k(t,y,s)=\psi(t,y,s) \hbox{ \rm  for each } (t,y,s)\in\bR\times\bR^n\times\bR,$$
   the assertion in
  Lemma \ref{lem ito foumul} still holds for  $\psi$.
\end{rmk}

 Rewritting  \eqref{eq bspde in ito formula} into
\begin{equation*}
  \begin{split}
    u(t,x)=u(0,x)+\int_0^t\left(\Delta u(s,x)+\partial_{x_i} \tilde{h}^i(s,x)-h^0(s,x)\right)\, ds+\int_0^tz^r(s,x)\, dW_s^r
  \end{split}
\end{equation*}
 with $$\tilde{h}^i(s,x):=-\partial_{x_i} u(s,x)-h^i(s,x),$$
   we obtain
\begin{lem}\label{rmk ito formula}
    Let all the assumptions on $\phi$ of Lemma \ref{lem ito foumul} be satisfied and \eqref{eq bspde in ito formula} hold in the weak sense of Definition \ref{definition of weak solution} with $u(T)\in L^2(\Omega,\sF_T,L^2(\cO))$, $z\in W^{2}_{\sF}(Q)$,  $h^i\in W^{2}_{\sF}(Q), i=1,\cdots,n$ and $h^0\in W^1_{\sF}(Q)$. We assume further that $\phi'(s,x,r)\leq M$ for any $(s,x,r)\in \bR\times\bR^n\times\bR$. If  $u\in\dot{W}^{1,2}_{\sF}(Q)\cap S^2(L^2(\cO))$, then \eqref{eq ito lemma} holds almost surely   for all $t\in[0,T]$.
\end{lem}
     The proof is very similar to that of \cite[Proposition 2]{DenisMatoussi2009} and is omitted here. The only difference is that to prove Lemma \ref{rmk ito formula} we use Lemma \ref{lem ito foumul} instead of \cite[Lemma 7]{DenisMatoussiStoica2005}.

%  Lemma \ref{lem ito foumul},
% both  It\^o formulas in \cite[Proposition 2]{DenisMatoussi2009} and
% can be extended to our test function $\phi$ and equation \eqref{eq bspde in ito formula},
%   but with the assumptions on the regularity and growth of $\phi$ being suitably  modified.
%\end{rmk}
Through a standard procedure  we obtain by Lemma \ref{lem ito foumul} the following

\begin{lem}\label{lem ito for non null bdary condition}
%Let $\varphi:\bR\longrightarrow\bR$ be a twice differentiable function with bounded  continuous second order derivative and $\phi'(0)=0$.
Let all the assumptions on $\phi$ of Lemma \ref{lem ito foumul} be satisfied. If the function $u$ in   \eqref{eq bspde in ito formula}
belongs to $ W^{1,2}_{\sF}(Q)\cap S^2(L^2(Q))$ with $u^+\in \dot{W}^{1,2}_{\sF}(Q)$, we have almost surely
\begin{equation}
  \begin{split}
    &\int_{\cO}\phi(t,x,u^+(t,x))\, dx+\frac{1}{2}\int_t^T\ll\phi''(s,\cdot,u^+(s)),\,|z^u(s)|^2\gg\, ds\\
    =&
    \int_{\cO}\phi(T,x,u^+(T,x))\, dx-\int_t^T\!\!\!\int_{\cO}\partial_{s}\phi(s,x,u^+(s,x))\, dxds\\
    &+\int_t^T\ll\phi(s,\cdot,u^+(s)),\, h^{0,u}(s)\gg\, ds               \\
    &-\int_{t}^T\ll\phi''(s,\cdot,u^+(s))\partial_{x_i} u^+(s)
    +\partial_{x_i} \phi'(s,\cdot,u^+(s)),\quad h^{i,u}(s)\gg\, ds\\
    &-\int_t^T\ll\phi'(s,\cdot, u^+(s)),\,\ z^{r,u}(s)\gg\, dW_s^r,\,\quad  t\in [0,T]
  \end{split}
\end{equation}
with
  $$h^{i,u}:=1_{\{u>0\}}h^i,\quad i=0,1,\cdots,n$$
and $$ z^{r,u}=1_{\{u>0\}}z^r,\quad r=1,\cdots,m; \quad z^u:=(z^{1,u}, \cdots, z^{m,u}).$$
\end{lem}

\begin{rmk} Note that the assumption $u^+\in \dot{W}_{\sF}^{1,2}(Q)$ does not imply that $u$ vanishes in a generalized sense on the boundary $\partial\cO$ and therefore
Lemma~\ref{lem ito foumul} can not be applied directly to get the corresponding equation~\eqref{eq bspde in ito formula} for $u^+$.
\end{rmk}

\begin{proof}[{\bf Sketch of the proof}]
\textbf{Step 1.} For $k\in\bN$, define
\begin{equation}\label{*}
  \psi(s)=\psi_k(s):=
  \left\{\begin{array}{l}
  \begin{split}
    0,\quad &s\in(-\infty,\frac{1}{k});\\
    \frac{k}{2}(s-\frac{1}{k})^2,\quad &s\in[\frac{1}{k},\frac{2}{k}];\\
    s-\frac{3}{2k},\quad &s\in (\frac{2}{k},+\infty).
    \end{split}
  \end{array}\right.
 \end{equation}
Then the assumptions on $u^+$ imply that $\psi(u)\in\dot{W}_{\sF}^{1,2}(Q)$.

Take $\varphi\in C_c^{\infty}(\cO)$ and set $\mathscr{V}:=\varphi
u$. Then $\mathscr{V}\in\dot{W}_{\sF}^{1,2}(Q)$. Since   \eqref{eq bspde in ito formula} holds in the weak sense of Definition \ref{definition of weak solution}, we have almost surely for any $\xi\in C_c^{\infty}(\cO)$
\begin{equation*}
  \begin{split}
    &\ll \xi,\ \varphi u(t) \gg\\
    =\ &\ll \xi,\ \varphi u(T)  \gg+\int_t^T\ll\xi,\ \varphi h^0(s)-\partial_{x_i}\varphi h^i(s)\gg ds   \\
    &-\int_t^T\ll\partial_{x_i}\xi,\ \varphi h^i(s) \gg ds
    -\int_t^T\ll \xi,\ \varphi z^r(s)\gg dW_s^r,\ \ \forall t\in[0,T].
  \end{split}
\end{equation*}
Hence, there holds
\begin{equation*}
  \begin{split}
    \mathscr{V}(t,x)=&\mathscr{V}(T,x)+\int_t^T\left[\varphi(x)h^0(s,x)
    -\partial_{x_i}\varphi(x) h^i(s,x)
    +\partial_{x_i}\left(\varphi (x)h^i(s,x)\right)\right]ds\\
    &-\int_t^T\varphi(x)z^r(s,x)dW^r_s,\quad t\in [0,T]
  \end{split}
\end{equation*}
in the weak sense of Definition \ref{definition of weak solution}.

For $\tilde{\varphi}\in C_c^{\infty}(O)$, by Lemma \ref{lem ito foumul} and Remark
\ref{rmk ito formula 0} we have almost surely
\begin{equation}\label{eq 1 in lem non null bdary ito}
  \begin{split}
    &\ll \psi(\mathscr{V}(t)),\ \ \tilde{\varphi}\gg
    +\frac{1}{2}\int_t^T\ll \psi''(\mathscr{V}(s))\tilde{\varphi},\quad |\varphi z(s)|^2\gg \, ds\\
    =&
    \ll \psi(\mathscr{V}(T)),\ \ \tilde{\varphi}\gg
    +\int_t^T\ll \psi'(\mathscr{V}(s))\tilde{\varphi},\quad \varphi h^0(s)\gg \, ds\\
    &
    -\int_t^T \ll \partial_{x_i} (\tilde{\varphi}\psi'(\mathscr{V}(s))\varphi),\ \, h^i(s)\gg ds\\
    &-\int_t^T\ll \psi'(\mathscr{V}(s))\tilde{\varphi},\ \, \varphi z^r(s)\gg \, dW^r_s,\quad \forall t\in[0,T].
  \end{split}
\end{equation}
Choosing $\varphi$ such that $\varphi\equiv 1$ in an open subset $\cO'\Subset \cO$ (i.e., $\overline{\cO'}\subset \cO$ ) and
$supp(\tilde{\varphi})\subset\cO'$, we have almost surely
\begin{equation*}
  \begin{split}
    &\ll\tilde{\varphi},\ \psi(u(t))\gg +\frac{1}{2}\int_t^T\ll \tilde{\varphi},\quad \psi''(u(s))|z(s)|^2\gg \, ds\\
    =&
    \ll \tilde{\varphi},\,\psi(u(T))\gg
    +\int_t^T\ll \tilde{\varphi},\ \ \psi'(u(s))h^0(s)\gg \, ds\\
    &
    -\int_t^T\ll \partial_{x_i} (\tilde{\varphi}\psi'(u(s))),
    \quad h^i(s)\gg \, ds-\int_t^T\ll \tilde{\varphi},\ \ \psi'(u(s))z^r(s)\gg\, dW^r_s \\
    =&
    \ll\tilde{\varphi} ,\,\psi(u(T)) \gg
    +\int_t^T\ll \tilde{\varphi},\ \ \psi'(u(s))h^0(s)\gg \, ds\\
    &-\int_t^T\ll \tilde{\varphi},\ \ \psi'(u(s))z^r(s)\gg\, dW^r_s -\int_t^T\ll \partial_{x_i} \tilde{\varphi},\quad \psi'(u(s))
    h^i(s)\gg \, ds\\
    & -\int_t^T\ll \tilde{\varphi},\ \ \psi''(u(s))\partial_{x_i} u(s)
    h^i(s)\gg \, ds,\quad \forall t\in [0,T].
  \end{split}
\end{equation*}
Since $\tilde{\varphi}$ is arbitrary, we have
\begin{equation}\label{Eq step1}
  \begin{split}
    \psi(u(t,x))=
    &\psi(u(T,x))+\int_t^T \!\!\psi'(u(s,x))h^0(s,x)\, ds
    -\frac{1}{2}\int_t^T\!\!\psi''(u(s,x))|z(s,x)|^2\, ds\\
    &
    -\int_t^T\psi'(u(s,x))z^r(s,x)\, dW^r_s-\int_t^T  \psi''(u(s,x))\partial_{x_i} u(s,x)
    h^i(s,x)\, ds\\
    &+\int_t^T  \partial_{x_i}
    (\psi'(u(s,x))h^i(s,x))\, ds
  \end{split}
\end{equation}
holds in the weak sense of Definition \ref{definition of weak solution}.

\textbf{Step 2.} It is sufficient to prove this lemma for test functions $\phi$ of bounded first and second derivatives. Since \eqref{Eq step1} holds for $\psi=\psi_k$, $k=1,2,\cdots$, in view of Lemma \ref{rmk ito formula} we obtain
\begin{equation*}
  \begin{split}
    &\int_{\cO}\phi(t,x,\psi_k(u(t,x)))\, dx+\frac{1}{2}\int_t^T\ll \phi''(s,\cdot,\psi_k(u(s))),\ \ |\psi_k'(u(s))z^u(s)|^2\gg \, ds\\
    =&
    \int_{\cO}\phi(T,x,\psi_k(u(T,x)))\, dx
    +\int_t^T\ll \phi'(s,\cdot,\psi_k(u(s))),\ \ \psi_k'(u(s))h^{0,u}(s)\gg \, ds               \\
    &-\int_t^T\!\!\!\int_{\cO}\partial_{s}\phi(s,x,\psi_k(u(s,x)))\, dxds
    -\frac{1}{2}\int_t^T \ll \phi'(s,\cdot,\psi_k(u(s))),\ \ \psi_k''(u(s))|z(s)|^2      \gg \, ds       \\
    &-\int_{t}^T \ll \phi'(s,\cdot,\psi_k(u(s))),\quad \psi_k''(u(s))\partial_{x_i} u(s)
    h^i(s)  \gg \, ds \\
    &-\int_{t}^T\ll \phi''(s,\cdot,\psi_k(u(s)))\psi_k'(u(s))\partial_{x_i} u(s)\\
    &~~~~~~~~~~~~~~~~~
    +\partial_{x_i} \phi'(s,\cdot,\psi_k(u(s))),\quad \psi_k'(u(s))h^{i,u}(s)\gg \, ds\\
    &
    -\int_t^T\ll \phi'(s,\cdot,\psi_k(u(s))),\quad \psi_k'(u(s))z^{r,u}(s)\gg\, dW_s^r
  \end{split}
\end{equation*}
holds almost surely for all $t\in[0,T]$.  From properties of $\phi$, we have $\phi'(t,x,r)\leq M |r|$ for any
$(t,x,r)\in[0,T]\times\cO\times\bR$. It follows that for any $(s,x)\in [0,T]\times\cO$,
\begin{equation}
    \begin{split}
       \left|\phi'(s,x,\psi_k(u(s,x)))\psi_k''(u(s,x))\right|
    \leq&
         \ M\left|\psi_k(u(s,x))\right|
         \left| \psi_k''(u(s))   \right|
    \\
    =&
        \ \frac{Mk}{2} \left|u(s,x)-\frac{1}{k}\right|^2 k 1_{[\frac{1}{k},\frac{2}{k}]}(u(s,x))
    \\
    \leq&
         \ M 1_{[\frac{1}{k},\frac{2}{k}]}(u(s,x)).
    \end{split}
\end{equation}
On the other hand, we check that $\lim_{k\rightarrow \infty}\|\psi_k(u)-u^+\|_{W^{1,2}_{\sF}(Q)}=0$.
 Therefore, by the dominated convergence theorem and taking limits in $L^1([0,T]\times\Omega,\sP,\bR)$
 on both sides of the above equation, we prove our assertion.
\end{proof}

\section{Solvability of Equation (\ref{1.1})}
Before the solvability of equation \eqref{1.1}, we give a useful lemma which is borrowed from  \cite[Corollary B1]{DuffieEpsteinSDU92} and called the stochastic Gronwall-Bellman inequality.
\begin{lem}\label{lem sto Gronwall-bellman}
 Let
$(\Omega,\mathcal{F},\mathbb{F},P)$ be a filtered probability space whose
filtration $\mathbb{F}=\{ \mathcal{F}_t:t\in [0,T] \}$ satisfies the usual
conditions. Suppose $\{ Y_s\}$ and $\{ X_s\}$ are optional integrable
processes and $\alpha$ is a nonnegative constant. If for all $t$,
$s\rightarrow E[Y_s|\mathcal{F}_t]$ is continuous almost surely and
$Y_t\leq(\geq) E[\int_t^T(X_s+\alpha Y_s)ds|\mathcal{F}_t]+Y_T$, then for
all $t$,
$$ Y_t\leq (\geq)e^{\alpha(T-t)}E[Y_T|\mathcal{F}_t]+E\left[ \int_t^Te^{\alpha (s-t)}
 X_s ds|\mathcal{F}_t \right]\quad a.s..   $$
\end{lem}
\begin{thm}\label{thm solve (1.1)}
  Let assumptions $({\mathcal A}1)$--$({\mathcal A}3)$ be satisfied and $\left\{h^i,i=0,1,\cdots,n\right\}\subset \cM^2(Q)$.  Then  $\dot{\sU}\times\dot{\sV}(G,f+h,g+h^0)$ (with $h=(h^1,\cdots,h^n)$) admits one and only one element $(u,v)$ which satisfies
   the following estimate 
    \begin{equation}\label{estimates in thm general existence}
      \begin{split}
        \|u\|_{\cV_2(Q)}+\|v\|_{\cM^2(Q)}\leq C \left\{\|G\|_{L^{\infty}(\Omega,\sF_T,L^2(\cO))}
        +A_p(f_0,g_0)+H_2(h,h^0)  \right\},
      \end{split}
    \end{equation}
    where $C$ is a constant depending on $n,p,q,\kappa,\lambda,\beta,\varrho,\Lambda_0,T,|\cO|$ and $L$.
\end{thm}

\begin{proof}[{\bf Proof. }]
 \textbf{Step 1.}
 Let  $({\mathcal A}3)_0$ be satisfied. From \cite[Theorem 2.1]{QiuTangBDSDES2010},  there is a unique weak solution $(u,v)$ in the
  space $(\dot{W}^{1,2}_{\sF}(Q)\cap S^2(L^{2}(\cO)))\times W^{2}_{\sF}(Q)$.
%    First, we assume $(u,v)\in \dot{\sU}\times\dot{\sV}(G,f+h,g+h^0)$ and prove estimate \eqref{estimates in thm general existence}.
    
    $Claim~(*):~(u,v)\in \dot{\sU}\times\dot{\sV}(G,f+h,g+h^0)$. 
   
   We shall prove $Claim~(*)$ in \textbf{Step 2}.
   By Lemma \ref{lem ito foumul}, we have almost surely
  \begin{equation}
    \begin{split}
      &\|u(t)\|_{L^2(\cO)}^2+\int_t^T\|v(s)\|_{L^2(\cO)}^2\,ds\\
      =&\|G\|^2_{L^2(\cO)}
      +2\int_t^T\ll
      u(s),\ \,
            b^i\partial_{x_i}u(s)+c\,u(s)+\varsigma^r v^r(s)+h^0(s)
            \gg ds\\
        &-2\int_t^T\ll \partial_{x_j}  u(s), \ \ a^{ij}\partial_{x_i} u(s)
        +\sigma^{jr} v^r(s)\\
        &~~~~~~~~~~~~~~~~~~~~~~~~~~~+f^j(s,\cdot,u(s),\nabla u(s),v(s))+h^j(s)\gg \, ds \\
        &-2\int_t^T\ll u(s),\ \, v^r(s)\gg\, dW^r_s+2\int_t^T\ll u(s),\ \, g(s,\cdot,u(s),\nabla u(s),v(s))\gg \, ds
    \end{split}
  \end{equation}
  for all $ t\in [0,T]$.
  Therefore, we obtain that almost surely
    \begin{equation}
    \begin{split}
      &E\left[\|u(t)\|_{L^2(\cO)}^2+\int_t^T\|v(s)\|_{L^2(\cO)}^2ds|\sF_t\right]\\
      =
        &\ E\left[\|G\|^2_{L^2(\cO)}|\sF_t\right]+2E\left[\int_t^T \ll u(s),g(s,\cdot,u(s),\nabla u(s),v(s))\gg ds  \big|\sF_t\right]\\
        &+2E\left[\int_t^T\ll
      u(s),\ \,
            b^i\partial_{x_i}u(s)+c\,u(s)+\varsigma^r v^r(s)+h^0(s)
            \gg ds\big|\sF_t\right]\\
        &
        -2E\Big[\int_t^T\ll \partial_{x_j}  u(s),\ \, a^{ij}\partial_{x_i} u(s)+\sigma^{jr} v^r(s)\\
        &~~~~~~~+f^j(s,\cdot,u(s),\nabla u(s),v(s))+h^j(s)\gg ds\big|\sF_t\Big],\,\ \forall t\in[0,T].
    \end{split}
  \end{equation}
%  Denote $\cO_t:=[t,T]\times\cO$ for any $t\in[0,T)$.

  Using the Lipschitz condition and H\"older inequality, we get the following estimates
  \begin{equation}\label{eq est h}
    \begin{split}
      &
        2E\left[\int_t^T\left(\ll u(s),\ h^0(s)\gg-\ll\partial_{x_j}u(s),\ h^j(s)\gg\right) ds\big|\sF_t \right]
      \\
        \leq 
      &
        \ E\left[
        \int_t^T\!\left(\|u(s)\|^2_{L^2(\cO)}+\|h^0(s)\|_{L^2(\cO)}^2+\eps^{-1}\|h(s)\|^2_{L^2(\cO)}+\eps\|\nabla u(s)\|_{L^2(\cO)}^2  \right)ds
        \big|\sF_t\right],
    \end{split}
  \end{equation}
  \begin{equation}\label{eq est g}
    \begin{split}
        &
            \esssup_{\omega\in\Omega}\sup_{\tau\in[t,T]}
            2E\left[\int_\tau^T\ll u(s),\, g(s,\cdot,u(s),\nabla u(s),v(s))\gg ds\big|\sF_{\tau}\right]  \\
        \leq&
            \ \esssup_{\omega\in\Omega}\sup_{\tau\in[t,T]}
            2E\left[\int_\tau^T\ll u(s),\, g_0(s)+L(|u(s)|+|\nabla u(s)|+|v(s)|)\gg ds\big|\sF_{\tau}\right]\\
        \leq&
            \ \eps\|\nabla u\|_{2;\cO_t}^2+\eps_1\|v\|_{2;\cO_t}^2+C(\eps,\eps_1,L)\|u\|_{2;\cO_t}^2\\
        &
            +\esssup_{\omega\in\Omega}\sup_{\tau\in[t,T]}2E\left[\int_{\tau}^T\ll|u(s)|, \,|g_0(s)| \gg ds\big|\sF_{\tau}\right]\\
        \leq&
            \ \eps\|\nabla u\|_{2;\cO_t}^2+\eps_1\|v\|_{2;\cO_t}^2+C(\eps,\eps_1,L)\|u\|_{2;\cO_t}^2
            +2|\cO_t|^{\frac{1}{2}-\frac{1}{p}}     \|g_0\|_{\frac{p(n+2)}{n+2+p};\cO_t}  \|u\|_{\frac{2(n+2)}{n};\cO_t}\\
        \leq&
            \ \eps\|\nabla u\|_{2;\cO_t}^2+\eps_1\|v\|_{2;\cO_t}^2+C(\eps,\eps_1,L)\|u\|_{2;\cO_t}^2
            +c(n)|\cO_t|^{\frac{1}{2}-\frac{1}{p}}     \|g_0\|_{\frac{p(n+2)}{n+2+p};\cO_t}  \|u\|_{\cV_2(\cO_t)}
        \\
        \leq&
            \ \eps\|\nabla u\|_{2;\cO_t}^2+\eps_1\|v\|_{2;\cO_t}^2+C(\eps,\eps_1,L)\|u\|_{2;\cO_t}^2
            +\delta\|u\|_{\cV_2(\cO_t)}^2
        \\
        &
            +C(\delta,n,p,|Q|)\|g_0\|_{\frac{p(n+2)}{n+2+p};\cO_t}^2
    \end{split}
  \end{equation}
  and
  \begin{equation}\label{eq est linear part}
    \begin{split}
      &\esssup_{\omega\in\Omega}\sup_{\tau\in[t,T]}
      2E\left[\int_{\tau}^T\ll u(s),\ b^i(s)\partial_{x_i}u(s)+c(s)\,u(s)+\varsigma^r(s)v^r(s) \gg ds
                    \big|\sF_{\tau}\right]\\
      \leq &
      \ (\eps^{-1}+\eps^{-1}_1)\esssup_{\omega\in\Omega}\sup_{\tau\in[t,T]}
      E\left[\int_{\tau}^T\ll |b(s)|^2+|c(s)|+|\varsigma(s)|^2,\ u^2(s) \gg ds
                    \big|\sF_{\tau}\right]\\
      &+\eps \|\nabla u\|_{2;\cO_t}^2
      +\eps_1\|v\|_{2;\cO_t}^2\\
      \leq&
      \ \eps \|\nabla u\|_{2;\cO_t}^2
      +\eps_1\|v\|_{2;\cO_t}^2
      +(\eps^{-1}+\eps^{-1}_1)
      B_q(b,c,\varsigma)\|u\|^2_{\frac{2q}{q-1};\cO_t} \\
      \leq&
       \ \eps \|\nabla u\|_{2;\cO_t}^2
      +\eps_1\|v\|_{2;\cO_t}^2
      +(\eps^{-1}+\eps^{-1}_1)
      B_q(b,c,\varsigma)\|u\|^{2\alpha}_{\frac{2(n+2)}{n};\cO_t}\|u\|^{2(1-\alpha)}_{2;\cO_t} \\
      \leq&
      \ \eps \|\nabla u\|_{2;\cO_t}^2
      +\eps_1\|v\|_{2;\cO_t}^2\\
      &\ +(\eps^{-1}+\eps^{-1}_1)
      B_q(b,c,\varsigma)\left(C(n)\|u\|_{\cV_2(\cO_t)}\right)^{2\alpha}\|u\|^{2(1-\alpha)}_{2;\cO_t}
      \textrm{ (by Lemma \ref{lem emmbedding for space V})}\\
      \leq&
      \ \eps \|\nabla u\|_{2;\cO_t}^2
      +\eps_1\|v\|_{2;\cO_t}^2+\delta \|u\|_{\cV_2(\cO_t)}^2 \\
      &+
      C(\delta,n,q)\left|(\eps^{-1}+\eps_1^{-1})B_q(b,c,\varsigma)\right|^{\frac{1}{1-\alpha}}\|u\|_{2;\cO_t}^2
    \end{split}
  \end{equation}
  with
   $\alpha:=\frac{n+2}{2q}\in (0,1)$ and
    the three positive small parameters $\eps$, $\eps_1$ and $\delta$ waiting  to be determined later.
  Also, there exists a constant $\theta>\varrho'=\frac{\varrho}{\varrho-1}$ such that
  $\lambda-\kappa-\beta\theta>0$ and
  \begin{equation}\label{eq est principle part}
    \begin{split}
      -&E\left[\int_t^T2\ll \partial_{x_j}  u(s),a^{ij}\partial_{x_i} u(s)
      +\sigma^{jr}v^r(s)+f^j(s,u(s),\nabla u(s),v(s))\gg ds   \big|\sF_t\right]\\
      \leq&
      -E\left[\int_t^T\ll \partial_{x_j}  u(s),\,
                    (2a^{ij}(s)-\varrho\sigma^{jr}(s)\sigma^{ir}(s))\partial_{x_i} u(s)\gg ds
                    \big|\sF_t\right]\\
                    & +\frac{1}{\varrho}E\left[\int_t^T\|v(s)\|^2_{L^2(\cO)}\,ds\big|\sF_t\right]\\
      &
      +2E\left[
            \int_t^T \ll |\nabla u(s)|,\,L|u(s)|+\frac{\kappa}{2}|\nabla u(s)|+\beta^{\frac{1}{2}}|v(s)|+|f_0(s)|\gg ds
           \big|\sF_t\right]\\
      \leq&
      -(\lambda-\kappa-\beta\theta-\eps)E\left[\int_t^T
                    \|\nabla u(s)\|_{L^2(\cO)}^2\,ds
                    \big|\sF_t\right]
      +C(\eps)\|f_0\|^2_{2;\cO_t}\\
      &+\left(\frac{1}{\varrho}+\frac{1}{\theta}\right)E\left[\int_t^T\|v(s)\|^2_{L^2(\cO)}\,ds\big|\sF_t\right]
      +C(\eps,L)E\left[\int_t^T\| u(s)\|^2_{L^2(\cO)}\,ds\big|\sF_t\right]\\
      \leq&
      -(\lambda-\kappa-\beta\theta-\eps)E\left[\int_t^T
                    \|\nabla u(s)\|_{L^2(\cO)}^2\,ds
                    \big|\sF_t\right]\\
    &+C(\eps,|Q|,p,L)\left\{E\left[\int_t^T\| u(s)\|^2_{L^2(\cO)}\,ds\big|\sF_t\right]+\|f_0\|_{p;\cO_t}^2\right\}\\
      &
      +\left(\frac{1}{\varrho}+\frac{1}{\theta}\right)E\left[\int_t^T\|v(s)\|^2_{L^2(\cO)}\,ds\big|\sF_t\right]
      ,\ \,\forall t\in[0,T]\ \,a.s..
    \end{split}
  \end{equation}
  Choosing $\eps$ and $\eps_1$ to be small enough, we get
  \begin{equation*}
    \begin{split}
      &\|u\|_{\cV_2(\cO_t)}^2+\|v\|_{2;\cO_t}^2\\
      \leq &
        \ 3~ \esssup_{\omega\in\Omega}\sup_{\tau\in[t,T]}\left\{
            \|u(\tau)\|_{L^2(\cO)}^2+E\left[ \int_{\tau}^T(\|\nabla u(s)\|_{L^2(\cO)}^2+\|v(s)\|_{L^2(\cO)}^2)  \,ds\big|\sF_{\tau}\right]\right\}\\
      \leq&
        \ C_1\biggl\{\|G\|_{L^{\infty}(\Omega,\sF_T,L^2(\cO))}^2
            +\|f_0\|_{p;\cO_t}^2+\left|H_2(h,h^0)\right|^2\\
    &+\delta\|u\|_{\cV_2(\cO_t)}^2+C(\delta,n,q,\Lambda_0)\int_t^T\|u(s)\|_{\cV_2(\cO_s)}^2\, ds
            +C(\delta,n,p,|Q|)\|g_0\|_{\frac{p(n+2)}{n+2+p};\cO_t}^2 \biggr\}
    \end{split}
  \end{equation*}
  with  the constant $C_1$ being independent of $\delta$.
  Then by choosing $\delta$ to be so small that $C_1\delta<1/2$, we obtain
  \begin{equation}\label{eq gronwall in solve}
    \begin{split}
      &\|u\|_{\cV_2(\cO_t)}^2+\|v\|_{2;\cO_t}^2\\
      \leq &
        \ C\left\{\|G\|_{L^{\infty}(\Omega,\sF_T,L^2(\cO))}^2
            +\int_t^T\|u(s)\|_{\cV_2(\cO_s)}^2\,ds+\left|A_p(f_0,g_0)\right|^2
            +\left|H_2(h,h^0)\right|^2 \right\}.
    \end{split}
  \end{equation}
  Thus, it follows from Gronwall inequality  that
  \begin{equation}\label{estim 1 in thm solv (1.1)}
    \|u\|_{\cV_2(\cO_t)}^2+\|v\|_{2;\cO_t}^2
      \leq
        C\left\{\|G\|_{L^{\infty}(\Omega,\sF_T,L^2(\cO))}^2
            +\left|A_p(f_0,g_0)\right|^2+\left|H_2(h,h^0)\right|^2 \right\}
  \end{equation}
  with the constant $C$ depending on $T,L,\Lambda_0,\lambda,\beta,\kappa,\varrho,n,p,q$ and $|Q|$.

\textbf{Step 2.} We prove $Claim~(*)$. It is sufficient to prove $(u,v)\in\dot{\cV}_{2,0}(Q)\times\cM^2(Q)$. Making estimates like \eqref{eq est h} and \eqref{eq est principle part}, we obtain
\begin{equation}
  \begin{split}
        &
         \|u(t)\|_{L^2(\cO)}^2+E\left[\int_t^T\|v(s)\|_{L^2(\cO)}^2ds|\sF_t\right]\\
      =
        &
         \ E\left[\|G\|^2_{L^2(\cO)}|\sF_t\right]+2E\left[\int_t^T \ll u(s),g(s,\cdot,u(s),\nabla u(s),v(s))\gg ds  \big|\sF_t\right]
        \\
        &
         +2E\left[\int_t^T\ll
         u(s),\ \,
            b^i\partial_{x_i}u(s)+c\,u(s)+\varsigma^r v^r(s)+h^0(s)
            \gg ds\big|\sF_t\right]
        \\
        &
         -2E\Big[\int_t^T\ll \partial_{x_j}  u(s),\ \, a^{ij}\partial_{x_i} u(s)+\sigma^{jr} v^r(s)\\
        &
         ~~~~~~~+f^j(s,\cdot,u(s),\nabla u(s),v(s))+h^j(s)\gg ds\big|\sF_t\Big]
        \\
    \leq
        &
         \ -(\lambda-\kappa-\beta\theta-\eps)E\left[\int_t^T
                    \|\nabla u(s)\|_{L^2(\cO)}^2\,ds
                    \big|\sF_t\right]
        \\
        &
         +\left(\frac{1}{\varrho}+\frac{1}{\theta}+\eps\right)
                    E\left[\int_t^T\|v(s)\|^2_{L^2(\cO)}\,ds\big|\sF_t\right]
        \\
        &
         +E\left[\|G\|^2_{L^2(\cO)}|\sF_t\right]
         +C(\eps)\left(\left|H_2(f_0,g_0)\right|^2+\left|H_2(h,h^0)\right|^2\right)
        \\
        &   +C(\eps,\lambda,\beta,\kappa,\varrho,L,\||b|\|_{\cL^{\infty}(Q)},
        \|c\|_{\cL^{\infty}(Q)},\||\varsigma|\|_{\cL^{\infty}(Q)})
         E\left[
         \int_t^T \|u(s)\|_{L^2(\cO)}^2 ds
         \big|\sF_t\right]
  \end{split}
\end{equation}
with the positive constant $\eps$ waiting to be determined later.
Letting $\eps$ be small enough, we have almost surely
\begin{equation*}
  \begin{split}
    &
        \|u(t)\|_{L^2(\cO)}^2+E\left[\int_t^T
        \left(\|\nabla u(s)\|_{L^2(\cO)}^2 + \|v(s)\|_{L^2(\cO)}^2\right)ds |\sF_t\right]
    \\
  \leq
    &
        \ C\left\{\|G\|^2_{L^{\infty}(\Omega,\sF_T,L^2(\cO))}
        + \left|H_2(f_0,g_0)\right|^2+\left|H_2(h,h^0)\right|^2
         +E\left[
         \int_t^T \|u(s)\|_{L^2(\cO)}^2 ds
         \big|\sF_t\right]
          \right\}
  \end{split}
\end{equation*}
for all $t\in[0,T]$. Then, by Lemma \ref{lem sto Gronwall-bellman} we obtain
\begin{equation*}
  \begin{split}
    &
       \esssup_{\omega\in\Omega}\sup_{t\in [0,T]}\left\{ \|u(t)\|_{L^2(\cO)}^2+E\left[\int_t^T
        \left(\|\nabla u(s)\|_{L^2(\cO)}^2 + \|v(s)\|_{L^2(\cO)}^2\right)ds |\sF_t\right]
        \right\}
    \\
  \leq
    &
        \ C\left\{\|G\|^2_{L^{\infty}(\Omega,\sF_T,L^2(\cO))}
        + \left|H_2(f_0,g_0)\right|^2+\left|H_2(h,h^0)\right|^2
          \right\}
  \end{split}
\end{equation*}
with the constant $C$ depending on $\lambda,\beta,\kappa,\varrho,L,T,\||b|\|_{\cL^{\infty}(Q)},
        \|c\|_{\cL^{\infty}(Q)},\||\varsigma|\|_{\cL^{\infty}(Q)}$.
Hence, $(u,v)\in\dot{\cV}_{2,0}(Q)\times\cM^2(Q)$. We complete the proof of $Claim~(*)$.

\textbf{Step 3.}
 Now we consider the general case of assumption~$({\mathcal A}3)$.
  The existence of the solution can be shown by approximation. As $p>n+2$ and $\cM^p(Q)\subset\cM^2(Q)$, $f_0\in \cM^2(Q)$.
  We  approximate the functions $b$, $c$, $\varsigma$ and $g$ by
  \begin{equation}\label{eq in thm solve (1.1) appxim g}
    b_k:=b 1_{\{|b|\leq k\}},\ c_k:=c 1_{\{|c|\leq k\}},\ \varsigma_k:=\varsigma 1_{\{|\varsigma|\leq k\}} \ {\rm and}\   g^k  :=g-g_0+g^k_0,
  \end{equation}
   with $g^k_0=g_0 1_{\{|g_0|\leq k\}}$.
    Then we have
    $$ \lim_{k\rightarrow \infty}B_q(b-b_k,c-c_k,\varsigma-\varsigma_k)+A_p(0,g_0-g_0^k)=0. $$
    Let $(u_k,v_k)\in \dot{\cV}_{2,0}(Q)\times \cM^2(Q)$ be the unique weak solution to   (\ref{1.1}) with
    $(b,c,\varsigma,f,g)$ being replaced by $(b_k,c_k,\varsigma_k,f+h,g^k+h^0)$.
    Then by estimate \eqref{estim 1 in thm solv (1.1)}, there exists a positive constant $C_0$ such that
    $$\sup_{k\in \bN}\left\{\|u_k\|_{\cV_2(Q)}^2+\|v_k\|_{2;Q}^2\right\} < C_0.$$
    For $k,l\in\bN$, the pair of random fields $(u_{kl},v_{kl}):=(u_k-u_l,v_k-v_l)\in \dot{\cV}_{2,0}(Q)\times \cM^2(Q)$ is the weak solution to the following BSPDE:
  \begin{equation*}
    (k,l)~~
    \left\{\begin{array}{l}
    \begin{split}
        -du_{kl}(t,x)=&\displaystyle \biggl[\partial_{x_j}\Bigl(a^{ij}(t,x)\partial_{x_i} u_{kl}(t,x)
        +\sigma^{jr}(t,x) v_{kl}^{r}(t,x)     \Bigr) +b_k^j(t,x)\partial_{x_j}u_{kl}(t,x)\\
        &\displaystyle
         +c_k(t,x)u_{kl}(t,x)+\varsigma^{r}_k(t,x)v^r_{kl}(t,x)
         \\
         &\displaystyle
         +b_{kl}^j(t,x)\partial_{x_j}u_{l}(t,x)
         +c_{kl}(t,x)u_l(t,x)+\varsigma_{kl}^r(t,x) v^r_l(t,x)\\
         &\displaystyle
         +\bar{g}_{kl}(t,x,u_{kl}(t,x),\nabla u_{kl}(t,x),v_{kl}(t,x))\\
        &\displaystyle +\partial_{x_j}\bar{f}_{kl}^j(t,x,u_{kl}(t,x),\nabla u_{kl}(t,x),v_{kl}(t,x))
                \biggr]\, dt\\ &\displaystyle
           -v_{kl}^{r}(t,x)\, dW_{t}^{r}, \quad
                     (t,x)\in Q:=[0,T]\times \mathcal {O};\\
        u_{kl}(T,x)=&0, \quad x\in\cO
        \end{split}
    \end{array}\right.
    \end{equation*}
    with
    \begin{equation*}
        \begin{split}
        \bar{f}_{kl}(t,x,R,Y,Z):=&f(t,x,R+u_l(t,x),Y+\nabla u_l(t,x),Z+v_l(t,x))\\
                        &-f(t,x,u_l(t,x),\nabla u_l(t,x),v_l(t,x)),\\
        \bar{g}_{kl}(t,x,R,Y,Z):=&g^k(t,x,R+u_l(t,x),Y+\nabla u_l(t,x),Z+v_l(t,x))\\
                        &-g^l(t,x,u_l(t,x),\nabla u_l(t,x),v_l(t,x)),\\
        (b_{kl},\ c_{kl},\ \varsigma_{kl})(t,x)
                        :=&
                        (b_k-b_l,\ c_k-c_l,\ \varsigma_k-\varsigma_l)(t,x).
     \end{split}
    \end{equation*}
    Since
   \begin{equation*}
    \begin{split}
      &\esssup_{\omega\in\Omega}\sup_{\tau\in[t,T]}
      2E\left[\int_{\tau}^T\ll u_{kl}(s),\ b_{kl}^i\partial_{x_i}u_l(s)
      +c_{kl}\,u_l(s)+\varsigma^r_{kl}v_l^r(s) \gg ds
                    \big|\sF_{\tau}\right]\\
      \leq &
      \ 2{\bar{\eps}}^{-1}
      \esssup_{\omega\in\Omega}\sup_{\tau\in[t,T]}
      E\left[\int_{\tau}^T\!\! \ll \left|b_{kl}(s)\right|^2
      +|c_{kl}(s)|+\left|\varsigma_{kl}(s)\right|^2,\ u_{kl}^2(s) \gg ds
                    \big|\sF_{\tau}\right]\\
      &+\bar{\eps} \left(\|\nabla u_l\|_{2;\cO_t}^2
      +\|v_l\|_{2;\cO_t}^2\right)\\
      \leq&
      \ \bar{\eps}\left(\|u_l\|_{\cV_2(Q)}^2+\|v_l\|_{2;Q}^2  \right)
      +2\bar{\eps}^{-1}
      B_q(b_{kl},c_{kl},\varsigma_{kl})\|u_{kl}\|^2_{\frac{2q}{q-1};\cO_t} \\
      \leq&
      \ \bar{\eps} C_0
      +\delta \|u_{kl}\|_{\cV_{2}(\cO_t)}^2+
      C(\delta,n,q)\left|\bar{\eps}^{-1} B_q(b_{kl},c_{kl},\varsigma_{kl})\right|^{\frac{2q}{2q-n-2}}\|u_{kl}\|_{2;\cO_t}^2
      \textrm{ (by Lemma \ref{lem emmbedding for space V})},
    \end{split}
  \end{equation*}
    in a similar way to the derivation of \eqref{eq gronwall in solve},  we obtain
    \begin{equation*}
      \begin{split}
        &\|u_{kl}\|_{\cV_2(\cO_t)}^2+\|v_{kl}\|_{2;\cO_t}^2\\
        \leq&
          \ C\bigg\{ \bar{\eps}+\left|A_p(0,g_0^k-g_0^l)\right|^2+\left(1+  \left|\bar{\eps}^{-1} B_q(b_{kl},c_{kl},\varsigma_{kl})\right|^{\frac{2q}{2q-n-2}}  \right)
          \int_t^T\|u_{kl}(s)\|_{\cV_2(\cO_s)}^2\,ds      \bigg\}
      \end{split}
    \end{equation*}
     which, by Gronwall inequality, implies
    \begin{equation}
      \begin{split}
        &\|u_{kl}\|_{\cV_2(Q)}^2+\|v_{kl}\|_{2;Q}^2\\
         \leq\ &C\left( \bar{\eps}+\left|A_p(0,g_0^k-g_0^l)\right|^2\right)
        \exp{\left[T \left(1+  \left|\bar{\eps}^{-1} B_q(b_{kl},c_{kl},\varsigma_{kl})\right|^{\frac{2q}{2q-n-2}}  \right)\right]}
      \end{split}
    \end{equation}
    with the constant $C$ being independent of $k$, $l$ and $\bar{\eps}$.
    By choosing  $\bar{\eps}$ to be small and then $k$ and $l$ to be sufficiently large, we conclude that
    $(u_k,v_k)$ is a Cauchy sequence in $\dot{\cV}_{2,0}(Q)\times\cM^2(Q)$. Passing to the limit, we
    check that the limit $(u,v)\in \dot{\sU}\times\dot{\sV}(G,f+h,g+h^0)$. In view of estimate \eqref{estim 1 in thm solv (1.1)} we prove estimate \eqref{estimates in thm general existence}.

  \tbf{Step 4.}
  It remains to prove the uniqueness. Assume that $(u',v')$ and $(u,v)$ are two weak solutions in $\dot{\cV}_{2,0}(Q)\times\cM^2(Q)$. Then their difference $(\bar{u},\bar{v}):=(u-u',v-v')\in\dot{\sU}\times\dot{\sV}(0,\bar{f},\bar{g})$ with
  \begin{equation*}
    \begin{split}
      \bar{f}(t,x,R,Y,Z):=&f(t,x,R+u'(t,x),Y+\nabla u'(t,x),Z+v'(t,x))\\
                        &-f(t,x,u'(t,x),\nabla u'(t,x),v'(t,x)),\\
      \bar{g}(t,x,R,Y,Z):=&g(t,x,R+u'(t,x),Y+\nabla u'(t,x),Z+v'(t,x))\\
                        &-g(t,x,u'(t,x),\nabla u'(t,x),v'(t,x)).
    \end{split}
  \end{equation*}
  Since $\bar{f}_0=0,\ \bar{g}_0=0$ and $\bar{u}(T)=0$,  we deduce from \eqref{estim 1 in thm solv (1.1)} that $\bar{u}=0$ and $\bar{v}=0$. The proof is complete.
\end{proof}

\begin{rmk}
  On the basis of the  monotone operator theory, Qiu and Tang in \cite{QiuTangBDSDES2010}
  established a theory of solvability for quasi-linear BSPDEs in an abstract framework.
  However even for the linear case $(f,g)\equiv (f_0,g_0)$,
  our BSPDE (\ref{1.1}) under assumptions $({\mathcal A}1)$--$({\mathcal A}3)$ falls beyond the framework of Qiu and Tang~\cite{QiuTangBDSDES2010}
  since our $b, c,$ and $\varsigma$ may be unbounded.
\end{rmk}

\begin{cor}\label{cor after thm solve (1.1)}
  Let assumptions $({\mathcal A}1)$--$({\mathcal A}3)$ be true,  $\left\{h^i,i=0,1,\cdots,n\right\}\subset \cM^2(Q)$ and $(u,v)\in \dot{\sU}\times\dot{\sV}(G,f+h,g+h^0)$ with $h=(h^1,\cdots,h^n)$. Let $\phi:\bR\times\bR^n\times\bR\longrightarrow\bR$ satisfy the assumptions of Lemma \ref{lem ito foumul}. Then we have almost surely
  \begin{equation*}
  \begin{split}
    &\int_{\cO}\phi(t,x,u(t,x))\,dx
    +\frac{1}{2}\int_t^T\ll \phi''(s,\cdot,u(s)),\ |v(s)|^2\gg ds\\
    =&
    \int_{\cO}\phi(T,x,G(x))\,dx
        -\int_t^T\int_{\cO}\partial_{s}\phi(s,x,u(s,x))\,dxds\\
    &+\int_t^T\ll \phi'(s,\cdot,u(s)),\ b^i\partial_{x_i}u(s)+c\,u(s)+\varsigma^r v^r(s)
    +h^0(s)\gg ds\\
    &+\int_t^T\ll \phi'(s,\cdot,u(s)),\ g(s,\cdot,u(s),\nabla u(s),v(s))\gg ds               \\
    &-\int_{t}^T \ll \phi''(s,\cdot,u(s))\partial_{x_i} u(s)
    +\partial_{x_i} \phi'(s,\cdot,u(s)),\quad
    a^{ji}\partial_{x_j}  u(s)
    +\sigma^{ri}v^{r}(s)\\
    &\quad \quad+f^{i}(s,\cdot,u(s),\nabla u(s),v(s))
    +h^i(s) \gg  ds\\
    &-\int_t^T\ll \phi'(s,\cdot,u(s)),\ v^{r}(s)\gg dW_s^r,\, \ \forall t\in [0,T].
  \end{split}
  \end{equation*}
\end{cor}
The proof of the corollary is rather standard and is sketched below.
\begin{rmk}\label{rmk after cor of solve (1.1)}
  In a similar way to Remark \ref{rmk ito formula 0}, our corollary also holds for $\psi$ in Remark \ref{rmk ito formula 0}.
\end{rmk}
\begin{proof}[{\bf Sketch of the proof}]
  First, one can check that all the terms involved in our assertion is well defined.
  Similar to the proof of Theorem \ref{thm solve (1.1)}, we still approximate $(b,c,\varsigma,g)$ by $(b_k,c_k,\varsigma_k, g^k  )$ which is defined in  \eqref{eq in thm solve (1.1) appxim g}.
   By Theorem \ref{thm solve (1.1)}, there is a unique weak solution $(u_k,v_k)$ to   (\ref{1.1}) with $(b,c,\varsigma,f,g)$ being replaced by $(b_k,c_k,\varsigma_k,f+h, g^k+h^0  )$.
   Then by Lemma \ref{lem ito foumul}, we have for each $k\in\bN$,
   \begin{equation}\label{eq for  g^k   in thm solve (1.1)}
    \begin{split}
        &\int_{\cO}\phi(t,x,u_k(t,x))\,dx
        +\frac{1}{2}\int_t^T\ll \phi''(s,\cdot,u_k(s)),\,|v_k(s)|^2\gg ds\\
        =&
        \int_{\cO}\phi(T,x,G(x))\,dx
            -\int_t^T\!\int_{\cO}\partial_{s}\phi(s,x,u_k(s,x))\,dxds\\
        &+\int_t^T\ll \phi'(s,\cdot,u_k(s)),\ b_k^i\partial_{x_i}u_k(s)+c_k\,u_k(s)+\varsigma_k^rv_k^r(s)
        +h^0(s)\gg ds\\
        &+\int_t^T\ll \phi'(s,\cdot,u_k(s)),\, g^k  (s,\cdot,u_k(s),\nabla u_k(s),v_k(s))\gg ds\\
        &-\int_t^T\ll \phi'(s,\cdot,u_k(s)),\,v_k^{r}(s)\gg dW_s^r               \\
        &-\int_{t}^T\ll \phi''(s,\cdot,u_k(s))\partial_{x_i} u_k(s)
        +\partial_{x_i} \phi'(s,\cdot,u_k(s)),\quad
        a^{ji}\partial_{x_j}  u_k(s)
        +\sigma^{ri}v_k^{r}(s)\\
        &\quad \quad+f^{i}(s,\cdot,u_k(s),\nabla u_k(s),v_k(s))+h^i(s)\gg ds
    \end{split}
  \end{equation}
 almost surely for all $t\in[0,T]$.
   On the other hand, from the proof of Theorem \ref{thm solve (1.1)} it follows that
   $$\lim_{k\rightarrow \infty}\left\{\|u-u_k\|_{\cV_2(Q)}+\|v-v_k\|_{2;Q}     \right\}=0.$$
   Hence passing to the limit in $L^1(\Omega,\sF)$ and taking into account the path-wise continuity of $u$, we prove our assertion.
\end{proof}

 We have

\begin{prop}\label{propsition ito non null bdary}
      Let assumptions $({\mathcal A}1)$--$({\mathcal A}3)$ be satisfied,  $\left\{h^i,i=0,1,\cdots,n\right\}\subset \cM^2(Q)$ and $(u,v)\in\sU\times\sV(G,f+h,g+h^0)$
      with $h=(h^1,\cdots,h^n)$ and $u^+\in \dot{\cV}_{2,0}(Q)$. Let $\phi:\bR\times\bR^n\times\bR\longrightarrow\bR$ satisfy the assumptions of Lemma \ref{lem ito foumul}.
      Then, with probability 1, the following relation
\begin{equation*}
  \begin{split}
    &\int_{\cO}\phi(t,x,u^+(t,x))\,dx+\frac{1}{2}\int_t^T\ll \phi''(s,\cdot,u^+(s)),\, |v^u(s)|^2\gg ds\\
    =&
    \int_{\cO}\phi(T,x,G^+(x))\,dx-\int_t^T\int_{\cO}\partial_{s}\phi(s,x,u^+(s,x))\,dxds
                   \\
    &-\int_{t}^T\ll \phi''(s,\cdot,u^+(s))\partial_{x_i} u^+(s)
    +\partial_{x_i} \phi'(s,\cdot,u^+(s)),\quad
    a^{ji}(s)\partial_{x_j}  u^+(s)\\
   &\quad\quad +\sigma^{ri}(s)v^{r,u}(s)+f^{i,u}(s)\gg \, ds\\
   &+\int_t^T\ll
        \phi'(s,\cdot,u^+(s)),\ b^i(s)\partial_{x_i}u^+(s)+c(s)\,u^+(s)+\varsigma^r(s)v^{r,u}(s)\gg ds\\
    &+\int_t^T\ll \phi'(s,\cdot,u^+(s)),\, g^{u}(s)\gg ds
    -\int_t^T\ll \phi'(s,\cdot,u^+(s)),\,v^{r,u}\gg dW_s^r
  \end{split}
\end{equation*}
holds almost surely for all $t\in[0,T]$
where
  \begin{equation*}
    \begin{split}
        &g^{u}(s,x):=1_{\{(s,x):u(s,x)>0\}}(s,x)\left(h^0(s,x)+g(s,x,u(s,x),\nabla u(s,x),v(s,x))\right);\\
        &f^{i,u}(s,x):=1_{\{(s,x):u(s,x)>0\}}(s,x)\left(h^i(s,x)+f^i(s,x,u(s,x),\nabla u(s,x),v(s,x))\right),\\
        &\ i=0,1,\cdots,n;
            \end{split}
  \end{equation*}
  and
  $$
 v^u:=(v^{1,u},\cdots,v^{m,u}), \quad v^{r,u}(s,x):=1_{\{(s,x):u(s,x)>0\}}(s,x)v^r(s,x),\ r=1,\cdots,m.
  $$
\end{prop}
The proof  is very similar to that of Lemma \ref{lem ito for non null bdary condition} and is omitted here. The main difference lies in
\textbf{Step 1} where we use Corollary \ref{cor after thm solve (1.1)} and Remark \ref{rmk after cor of solve (1.1)} instead of Lemma \ref{lem
ito foumul} and Remark \ref{rmk ito formula 0}.
%\begin{rmk}\label{rmk after prop 4.3}
%  Since, for any $h\in\cM^{2}(Q)$ and $u\in \cV_{2,0}(Q)$ there holds
%    \begin{equation}
%    \begin{split}
%            \esssup_{\omega\in\Omega}\sup_{\tau\in[t,T]}E\left[\int_{\tau}^T\ll|u(s)|, \,|h(s)| \gg ds\big|\sF_{\tau}\right]
%        \leq
%           \ \|u\|_{2;\cO_t}^2 +\|h\|_{2;\cO_t}^2,
%    \end{split}
%  \end{equation}
%  in a similar way to Theorem \ref{thm solve (1.1)}, we prove that for any $h,l=(l^1,\cdots,l^n)\in\cM^2(Q)$, $\dot{\sU}\times\dot{\sV}(G,g+h,f+l)$ admits a unique element $(u,v)$ satisfying
%\begin{equation}
%      \begin{split}
%        \|u\|_{\cV_2(Q)}+\|v\|_{\cM^2(Q)}\leq C \left\{\|G\|_{L^{\infty}(\Omega,\sF_T,L^2(\cO))}
%        +A_p(f_0,g_0)+A_2(h,l)  \right\},
%      \end{split}
%    \end{equation}
%    where $C$ is a constant depending on $n,p,q,\kappa,\lambda,\beta,\varrho,\Lambda_0,T,|\cO|$ and $L$.
%    Moreover, we verify that Corollary \ref{cor after thm solve (1.1)} and Proposition \ref{propsition ito non null bdary} also be true with $(f,g)$ being replaced by $(f+h,g+l)$.
%\end{rmk}

\section{The maximum  principles}

\subsection{The global case}
\begin{thm}\label{thm max princip global 1}
  Let assumptions $({\mathcal A}1)$--$({\mathcal A}4)$  hold. Assume that $(u,v)\in\cV_{2,0}(Q)\times\cM^2(Q)$ is a weak solution of   (\ref{1.1}).
  Then we have
  \begin{equation}\label{estim thm max principle global 1}
    \esssup_{(\omega,t,x)\in\Omega\times Q}u(\omega,t,x)
    \,\leq C\left\{  \esssup_{(\omega,t,x)\in\Omega\times \partial_{\rm p}Q}u^+(\omega,t,x)
    +A_p(f_0,g_0^+)+\|u^+\|_{2;Q}\right\}
  \end{equation}
  where  $C$ is a constant  depending on
  $n,p,q,\kappa,\lambda,\beta,\varrho,\Lambda_0,L_0,T,|\cO|$ and $L$.
\end{thm}

\begin{rmk}
   By the inequality $\esssup_{(\omega,t,x)\in\Omega\times \partial_{\rm p}Q}u^+(\omega,t,x)\leq L_1$,
   we  mean that $(u-L_1)^+\in\dot{\cV}_{2,0}(Q)$ and with probability 1,
    for any $\zeta\in C_c^{\infty}(\cO)$, there holds
    $$
    \lim_{t\rightarrow T_-}\ll \zeta,\, (u(t)-L_1)^+\gg\,  =0.$$
\end{rmk}

\begin{rmk}
  In Theorem \ref{thm max princip global 1}, assume further that
  $$\esssup_{(\omega,t,x)\in\Omega\times \partial_{\rm p}Q}|u(\omega,t,x)|\,\leq L_1<\infty . $$
  We have
  $u\in \cL^{\infty}(Q)$ and
  \begin{equation}\label{estim rmk max principle global 1}
    \|u\|_{\infty;Q}
    \,\leq C\left\{  L_1
    +A_p(f_0,g_0)+\|u\|_{2;Q}\right\}
  \end{equation}
  where $C$ is a constant  depending on
  $n,p,q,\kappa,\lambda,\beta,\varrho,\Lambda_0,L_0,T,|\cO|$ and $L$.
\end{rmk}

We start the proof of Theorem \ref{thm max princip global 1}  with borrowing the following lemma either from \cite[Lemma 1.2, Chapter
6]{ChenyazheParabolic} or from \cite[Lemma 5.6, Chapter 2]{Ladyzhenskaia_68}.
\begin{lem}\label{lem degiorgi iteration}
  Let $\{a_k:k=0,1,2,\cdots\}$ be a sequence of nonnegative numbers satisfying
  $$a_{k+1}\leq C_0 b^k a_k^{1+\delta},\,k=0,1,2,\cdots$$
  where $b>1$, $\delta>0$ and $C_0$ is a positive constant. Then
  if $$a_0\leq \theta_0:=C_0^{-\frac{1}{\delta}}b^{-\frac{1}{\delta^2}},$$
  we have $\lim_{k\rightarrow\infty} a_k=0$.
\end{lem}

\begin{proof}[{\bf Sketch of the proof}]
  We  use the induction principle. It is sufficient to prove the following assertion:
  \begin{equation}\label{eq in lem degiorgi iteration}
    a_k\leq \frac{\theta_0}{\nu^k},\,k=0,1,2,\cdots,
  \end{equation}
   with the parameter $\nu>1$ waiting to be determined later. It is obvious for $k=0$ that \eqref{eq in lem degiorgi iteration} holds. Assume that  \eqref{eq in lem degiorgi iteration} holds for $k=r$. Then we have
  $$a_{r+1}
  \leq C_0 b^r a_r^{1+\delta}
  \leq C_0 b^r \left(\frac{\theta_0}{\nu^r}\right)^{1+\delta}
  = \frac{\theta_0}{\nu^{r+1}} \cdot \frac{C_0b^r\theta_0^{\delta}}{\nu^{r\delta-1}}.
  $$
  Taking $\nu=b^{\frac{1}{\delta}}>1$, we obtain
  $$a_{r+1}\leq \frac{\theta_0}{\nu^{r+1}}\cdot C_0 \nu\theta_0^{\delta}=\frac{\theta_0}{\nu^{r+1}}.$$
\end{proof}
\begin{cor}\label{cor degiorgi iteration}
  Let $\phi:[r_0,\infty)\longrightarrow\bR^+$ be a nonnegative and decreasing function. Moreover, there
  exist constants $C_1>0$, $\alpha>0$ and $\zeta>1$ such that
  for any $l>r>r_0$,
  $$ \phi(l)\leq \frac{C_1}{(l-r)^{\alpha}}\phi(r)^{\zeta}.
      $$
  Then for
   $$d\geq C_1^{\frac{1}{\alpha}}|\phi(r_0)|^{\frac{\zeta-1}{\alpha}}2^{\frac{\zeta}{\zeta-1}},$$
  we have
     $\phi(r_0+d)=0$.
\end{cor}

\begin{proof}[{\bf Sketch of the proof}]
  Define
  $$
      r_k:=r_0+d-\frac{d}{2^k},k=0,1,2,\cdots.
          $$
  Then
      $$\phi(r_{k+1})
      \leq \frac{C_1 2^{(k+1)\alpha}}{d^{\alpha}}\phi(r_k)^{\zeta}
      =\frac{C_12^{\alpha}}{d^{\alpha}}2^{k\alpha}\phi(r_k)^{\zeta}.$$
      In view of our assumption on $d$, since
       $$
      \phi(r_0)\leq \theta_0=\left(\frac{C_12^{\alpha}}{d^{\alpha}}\right)^{-\frac{1}{\zeta-1}} 2^{-\frac{\alpha}{(\zeta-1)^2}}
      =d^{\frac{\alpha}{\zeta-1}} C_1^{-\frac{1}{\zeta-1}} 2^{-\frac{\alpha \zeta}{(\zeta-1)^2}}
      ,$$
      we deduce from Lemma \ref{lem degiorgi iteration} that $\lim_{k\rightarrow \infty}\phi(r_{k+1})=0$.

\end{proof}

\begin{proof}[{\bf Proof of Theorem \ref{thm max princip global 1}}]
  Assume that $L_0=0$, or else we consider $\tilde{u}(t,x):=e^{L_0 t}u(t,x)$ instead of $u$.
  It is sufficient to prove our theorem for the case
  $$\esssup_{(\omega,t,x)\in\Omega\times \partial_{\rm p}Q}u^+(\omega,t,x) \, <\infty.$$
   Then for  $k\geq \esssup_{(\omega,t,x)\in\Omega\times \partial_{\rm p}Q}u^+(\omega,t,x)$, we have
  $(u-k,v)\in \sU\times\sV(G-k, f^k  , g^k  )$   with
  $$ ( f^k  , g^k  )(\omega,t,x,R,Y,Z):=(f,g)(\omega,t,x,R+k,Y,Z)+(0,c(\omega,t,x)k)
    $$
    for $(\omega,t,x,R,Y,Z)\in \Omega\times [0,T]\times\cO\times\bR\times\bR^n\times\bR^m$.
    From Proposition \ref{propsition ito non null bdary}, we have almost surely
    \begin{equation*}
      \begin{split}
      &\int_{\cO}|(u(t,x)-k)^+|^2\,dx
      +\int_t^T\|v_k(s)\|^2_{L^2(\cO)}\,ds\\
      =&
      -2\int_t^T\ll \partial_{x_j} (u(s)-k)^+,\\
      &~~~~~~~a^{ij}\partial_{x_i} u(s)+\sigma^{jr}v^r_k(s)+ (f^k)^j(s,\cdot,(u(s)-k)^+,\nabla u(s),v_k(s))  \gg ds\\
      &+2\int_t^T\ll (u(s)-k)^+,\ b^i\partial_{x_i}u(s)+c\,(u(s)-k)^+ +\varsigma^r v_k^r(s)\gg ds\\
      &+2\int_t^T\ll (u(s)-k)^+,\  g^k  (s,\cdot,(u(s)-k)^+,\,\nabla u(s),v_k(s))   \gg ds\\
      &-2\int_t^T\ll (u(s)-k)^+,\ v_k^r(s)\gg dW_s^r,\quad \forall \ t\in [0,T]
      \end{split}
    \end{equation*}
    with $v_k:=v 1_{u>k}$.
    Therefore, we have
    \begin{equation}\label{eq ito square in thm glob maxim}
      \begin{split}
        &\int_{\cO}|(u(t,x)-k)^+|^2\,dx
        +E\bigg[ \int_t^T\|v_k(s)\|^2_{L^2(\cO)}\,ds \big|\sF_t\bigg]
        \\
        =&-2E\biggl[\int_t^T \ll \partial_{x_j} (u(s)-k)^+,\quad
      a^{ij}\partial_{x_i} u(s)+\sigma^{jr}v^r_k(s)\\
      &\quad\quad+ (f^k)  ^j(s,\cdot,(u(s)-k)^+,\nabla u(s),v_k(s)) \gg ds     \big|\sF_t\bigg]\\
      &+2E\bigg[\int_t^T\ll (u(s)-k)^+,\ b^i\partial_{x_i}u(s)+c\,(u(s)-k)^+ +\varsigma^r v_k^r(s)\gg ds\big|\sF_t\bigg]\\
      &+2E\bigg[\int_t^T\ll (u(s)-k)^+,\,  g^k  (s,\cdot,(u(s)-k)^+,\nabla u(s),v_k(s)) \gg ds \big|\sF_t\biggr],\,a.s..
      \end{split}
    \end{equation}
    Note that
    \begin{equation}\label{eq 1 in thm glob max}
      \begin{split}
        &\esssup_{\Omega}\!\sup_{\tau\in[t,T]}2E\left[\int_{\tau}^T\ll (u(s)-k)^+,\,  g_0^k (s)\gg ds   \big|\sF_{\tau}\right]\\
        \leq &
        \ 2\|( g_0^k )^+\|_{\frac{p(n+2)}{n+2+p};\cO_t}\|(u-k)^+\|_{\frac{2(n+2)}{n};\cO_t}
        \left|\{u>k\}\right|_{\infty;\cO_t}^{\frac{1}{2}-\frac{1}{p}}~ \textrm{  (H$\ddot{\textrm{o}}$lder inequality)} \\
        \leq &
        \ \delta\|(u-k)^+\|^2_{\frac{2(n+2)}{n};\cO_t}
        +C(\delta)\|( g_0^k )^+\|^2_{\frac{p(n+2)}{n+2+p};\cO_t}
        \left|\{u>k\}\right|_{\infty;\cO_t}^{1-\frac{2}{p}}\\
        \leq &
        \ \delta\|(u-k)^+\|^2_{\frac{2(n+2)}{n};\cO_t}
        +C(\delta,p,L)\left(\left|A_p(f_0,g_0^+)\right|^2+k^2\right)
        \left|\{u>k\}\right|_{\infty;\cO_t}^{1-\frac{2}{p}},
      \end{split}
    \end{equation}
    \begin{equation}\label{eq es linear global}
    \begin{split}
      &\esssup_{\omega\in\Omega}\!\!\sup_{\tau\in[t,T]}
      \!\!2E\left[\int_{\tau}^T\!\!\!\!\! \ll (u(s)-k)^+,\ b^i\partial_{x_i}u(s)+c\,(u(s)-k)^++\varsigma^rv^r_k(s) \gg \!ds
                    \big|\sF_{\tau}\right]\\
      \leq &
      \ C(\eps)\esssup_{\omega\in\Omega}\sup_{\tau\in[t,T]}
      \!E\left[\int_{\tau}^T\!\!\!\! \ll |b(s)|^2+|c(s)|+|\varsigma(s)|^2,\ \left|(u(s)-k)^+\right|^2 \gg \!ds
                    \big|\sF_{\tau}\right]\\
      &+\eps \left(\|\nabla (u-k)^+\|_{2;\cO_t}^2
      +\|v_k\|_{2;\cO_t}^2\right)\\
      \leq&
      \ \eps \left(\|\nabla (u-k)^+\|_{2;\cO_t}^2
      +\|v_k\|_{2;\cO_t}^2\right)
      +C(\eps)
      \Lambda_0 \|(u-k)^+\|^2_{\frac{2q}{q-1};\cO_t} \\
      \leq&
      \ \eps \left(\|\nabla (u-k)^+\|_{2;\cO_t}^2
      +\|v_k\|_{2;\cO_t}^2\right)+\delta \|(u-k)^+\|_{\frac{2(n+2)}{n};\cO_t}^2\\
      &+
      C(\delta,n,q,\eps,\Lambda_0)\|(u-k)^+\|_{2;\cO_t}^2,
    \end{split}
  \end{equation}
   and
    \begin{equation}\label{eq 2 in thm glob max}
      \begin{split}
        &2E\left[\int_t^T\ll |\nabla (u(s)-k)^+|,\,| f_0^k (s)|\gg ds\Big|\sF_t\right]\\
        \leq &
        \ \eps E \left[\int_t^T\|\nabla (u(s)-k)^+\|^2_{L^2(\cO)}\,ds \Big|\sF_t\right]
        +C(\eps)E \left[ \int_t^T\| f_0^k 1_{u>k} (s)\|_{L^2(\cO)}^2\,ds \Big|\sF_t\right] \\
        \leq &
        \ \eps E \left[\int_t^T\|\nabla (u(s)-k)^+\|^2_{L^2(\cO)}\,ds \Big|\sF_t\right]
        +C(\eps)\left|\{u>k\}\right|_{\infty;\cO_t}^{1-\frac{2}{p}}\| f_0^k \|_{p;\cO_t}^2\\
        \leq &
        \ \eps E \left[\int_t^T\|\nabla (u(s)-k)^+\|^2_{L^2(\cO)}\,ds \Big|\sF_t\right]  \\
        &+C(\eps,p,L)\left(\left|A_p(f_0,g_0^+)\right|^2+k^2\right)
        \left|\{u>k\}\right|_{\infty;\cO_t}^{1-\frac{2}{p}},a.s.
      \end{split}
    \end{equation}
    where $\eps $ and $\delta$ are two positive parameters waiting to be determined later and
    $$\left|\{u>k\}\right|_{\infty;\cO_t}:=\esssup_{\Omega}\sup_{\tau\in[t,T]}
        E\left[ |Q_{\tau}\cap \{u>k\}|  \big| {\sF_{\tau}}\right].$$
    In a similar way  to \eqref{eq est g} and \eqref{eq est principle part} in the proof of Theorem \ref{thm solve (1.1)}, we  obtain from \eqref{eq 1 in thm glob max}
    and \eqref{eq 2 in thm glob max} that with probability 1, for all $t\in[0,T]$
    \begin{equation}\label{eq 3 in thm glob max}
      \begin{split}
        &-2E\bigg[\int_t^T \ll \partial_{x_j} (u(s)-k)^+,\quad
      a^{ij}\partial_{x_i} u(s)+\sigma^{jr}v^r_k(s)\\
      &\quad\quad + (f^k)  ^j(s,\cdot,(u(s)-k)^+,\,\nabla u(s),v_k(s))  \gg  ds     \big|\sF_t\bigg]\\
      \leq&
      -(\lambda-\kappa-\beta\theta-\eps)E\bigg[ \int_t^T\|\nabla(u(s)-k)^+\|^2_{L^2(\cO)}\,ds\big|\sF_t  \bigg]\\
      & +\left(\frac{1}{\varrho}+\frac{1}{\theta}\right)E\bigg[\int_t^T\|v_k(s)\|^2_{L^2(\cO)}\,ds \big|\sF_t\bigg]\\
      &+C(\eps,L)E\bigg[\int_t^T\|(u(s)-k)^+\|^2_{L^2(\cO)}\,ds  \big|\sF_t\bigg]\\
      &+2E\bigg[\int_t^T(|\nabla (u(s)-k)^+|,| f_0^k (s)|)\,ds     \big|\sF_t\bigg]                       \\
      \leq&
      -(\lambda-\kappa-\beta\theta-2\eps)E\bigg[ \int_t^T\|\nabla(u(s)-k)^+\|^2_{L^2(\cO)}\,ds\big|\sF_t  \bigg]\\
      &+\left(\frac{1}{\varrho}+\frac{1}{\theta}\right)E\bigg[\int_t^T\|v_k(s)\|^2_{L^2(\cO)}\,ds \big|\sF_t\bigg]\\
      &+C(\eps,L)E\bigg[\int_t^T\|(u(s)-k)^+\|^2_{L^2(\cO)}\,ds  \big|\sF_t\bigg]\\
      &+C(\eps,p,L)\left(\left|A_p(f_0,g_0^+)\right|^2+k^2\right)
        \left|\{u>k\}\right|_{\infty;\cO_t}^{1-\frac{2}{p}}
      \end{split}
    \end{equation}
    and
    \begin{equation}\label{eq 4 in thm glob max}
      \begin{split}
        &\esssup_{\Omega}\sup_{\tau\in[t,T]}2E\bigg[\int_{\tau}^T\ll (u(s)-k)^+,\, g^k  (s,\cdot,(u(s)-k)^+,\nabla u(s),v_k(s))\gg ds \big|\sF_{\tau}\bigg]\\
        &\leq
        \eps\|\nabla (u-k)^+\|_{2;\cO_t}^2+\eps_1\|v_k\|_{2;\cO_t}^2+C(\eps,\eps_1,L)\|(u-k)^+\|^2_{2;\cO_t}\\
        &~~+\esssup_{\Omega}\sup_{\tau\in[t,T]}2E\left[\int_{\tau}^T\ll (u(s)-k)^+,\, g_0^k (s)\gg ds   \big|\sF_{\tau}\right]\\
        &\leq
        \eps\|\nabla (u-k)^+\|_{2;\cO_t}^2+\eps_1\|v_k\|_{2;\cO_t}^2+C(\eps,\eps_1,L)\|(u-k)^+\|^2_{2;\cO_t}\\
        &~~
          +\delta \|(u-k)^+\|^2_{\frac{2(n+2)}{n};\cO_t}+C(L,\delta,p)\left(\left|A_p(f_0,g_0^+)\right|^2+k^2\right)
        \left|\{u>k\}\right|_{\infty;\cO_t}^{1-\frac{2}{p}}
      \end{split}
    \end{equation}
    where $\theta$, $\eps$, $\eps_1$ and $\delta$ are four positive parameters such that
    $$\theta>\frac{\varrho}{\varrho-1}>1,
    \frac{1}{\varrho}+\frac{1}{\theta}+\eps+\eps_1<1
    \textrm{ and } \lambda-\kappa-\beta\theta-4\eps>0.$$
    Combining \eqref{eq ito square in thm glob maxim}, \eqref{eq es linear global}, \eqref{eq 3 in thm glob max} and \eqref{eq 4 in thm glob max}, we have
    \begin{equation}\label{eq estim 1  thm glob max}
      \begin{split}
        &\|(u-k)^+\|^2_{\cV_2(\cO_t)}+\|v_k\|^2_{2;\cO_t}\\
        \leq&
            \ C\bigg\{C(\delta)\|(u-k)^+\|^2_{2;\cO_t}+\delta \|(u-k)^+\|^2_{\frac{2(n+2)}{n};\cO_t} \\
            & \ \ \ \ +C(\delta ) \left(\left|A_p(f_0,g_0^+)\right|^2+k^2\right)
        \left|\{u>k\}\right|_{\infty;\cO_t}^{1-\frac{2}{p}} \bigg\},
      \end{split}
    \end{equation}
    where $C$ is a constant independent of $t$ and $\delta$.

     By Lemma \ref{lem emmbedding for space V}, $\cV_{2,0}(\cO_t)$ is continuously embedded into $\cM^{\frac{2(n+2)}{n}}(\cO_t)$.
     That is $$ \|(u-k)^+\|_{\frac{2(n+2)}{n};\cO_t}\leq C \|(u-k)^+\|_{\cV_2(\cO_t)}. $$
     Therefore, choosing  $\delta$ to be small enough, we obtain
     \begin{equation*}
       \begin{split}
         \|(u-k)^+\|^2_{\frac{2(n+2)}{n};\cO_t}
         \leq&
            C\|(u-k)^+\|^2_{2;\cO_t}+C\left(\left|A_p(f_0,g_0^+)\right|^2+k^2\right)\left|\{u>k\}\right|_{\infty;\cO_t}^{1-\frac{2}{p}}\\
         \leq&
            C(|T-t||\cO|)^{\frac{2}{n+2}}\|(u-k)^+\|^2_{\frac{2(n+2)}{n};\cO_t}\\
            &+C\left(\left|A_p(f_0,g_0^+)\right|^2+k^2\right)\left|\{u>k\}\right|_{\infty;\cO_t}^{1-\frac{2}{p}}.\\
       \end{split}
     \end{equation*}
     Choosing $t_1\in[0,T)$ such that $C(|T-t_1||\cO|)^{\frac{2}{n+2}} \leq \frac{1}{2}$, we get
     \begin{equation*}
       \begin{split}
         \|(u-k)^+\|^2_{\frac{2(n+2)}{n};\cO_{t_1}}
         \leq \,
            C\left(\left|A_p(f_0,g_0^+)\right|^2+k^2\right)\left|\{u>k\}\right|_{\infty;\cO_{t_1}}^{1-\frac{2}{p}}
       \end{split}
     \end{equation*}
     where the constant $C$ does not depend on $t_1$.

     Define
     $$\psi:\bR\longrightarrow \bR,~~~\psi(h)=\left|\{u>h\}\right|_{\infty;\cO_{t_1}}.$$
     Since for any $h>k$,
     $$\|(u-k)^+\|^2_{\frac{2(n+2)}{n};\cO_{t_1}}
        \geq (h-k)^2 \left|\{u>h\}\right|_{\infty;\cO_{t_1}}^{\frac{n}{n+2}},$$
     taking $k\geq A_p(f_0,g_0^+)$ we have
     \begin{equation*}
       \begin{split}
         \psi(h)^{\frac{n}{n+2}}\leq \frac{Ck^2}{(h-k)^2}\psi(k)^{1-\frac{2}{p}}
       \end{split}
     \end{equation*}
     which implies
     \begin{equation}\label{eq 5 in thm glob max}
       \begin{split}
         \psi(h)\leq \frac{Ck^{\alpha}}{(h-k)^{\alpha}}\psi(k)^{1+\bar{\eps}}
       \end{split}
     \end{equation}
     where $\alpha=\frac{2(n+2)}{n}$ and $\bar{\eps}=\frac{2(p-n-2)}{pn}>0$.
     Take $k_l=k(2-2^{-l})$, $l=0,1,2,\cdots.$
     Then from
     $$\psi(k_{l+1})\leq \frac{Ck_l^{\alpha}}{(k_{l+1}-k_l)^{\alpha}}\psi(k_l)^{1+\bar{\eps}},$$
     it follows that
     \begin{equation*}
       \begin{split}
         \psi(k_{l+1})
         \leq
            \hat{C} 2^{\alpha(l+1)}\psi(k_l)^{1+\bar{\eps}}.
       \end{split}
     \end{equation*}
%     where $\hat{C}=\hat{C}(\kappa,\lambda,\beta,\varrho,p,n)$.
    By Lemma \ref{lem degiorgi iteration}, there exists a constant $\theta_0=\theta_0(\hat{C},\bar{\eps})>0$, such that
    if $\psi(k_0)\leq \theta_0$, $\lim_{l\rightarrow\infty}\psi(k_l)=0$.
    Note that $k_0=k$ and $\psi(k_0)=\left|\{u>k\}\right|_{\infty;\cO_{t_1} }$.

    Taking
    $$k=\esssup_{(\omega,s,x)\in\Omega\times \partial_{\rm p}Q}u^+(\omega,s,x)
    +A_p(f_0,g_0^+)+{\theta_0^{-\frac{1}{2}}}\|u^+\|_{2;\cO_{t_1}},$$
     we have
    \begin{equation*}
      \begin{split}
        k^2
        \geq \frac{1}{\theta_0}\|u^+\|^2_{2;\cO_{t_1}}
        \geq \frac{1}{\theta_0} k^2 \left|\{u>k\}\right|_{\infty;\cO_{t_1}}
      \end{split}
    \end{equation*}
    which implies
    \begin{equation*}
      \begin{split}
        \psi(k_0)=\left|\{u>k\}\right|_{\infty;\cO_{t_1}}\leq \theta_0.
      \end{split}
    \end{equation*}
     Hence, $\psi(k_{\infty})=0$. Since $k_{\infty}=2k$, we obtain
     \begin{equation*}
    \esssup_{(\omega,s,x)\in\Omega\times \cO_{t_1}}\!\!\!u(\omega,s,x)\leq 2k
    = 2\left\{\!  \esssup_{(\omega,s,x)\in\Omega\times \partial_{\rm p}Q}  u^+(\omega,s,x)
    +A_p(f_0,g_0^+)+{\theta_0^{-\frac{1}{2}}}\|u^+\|_{2;Q}\right\}.
  \end{equation*}
  As $T-t_1$ only depends on the structure terms like $n,\lambda,\kappa,\beta,\varrho,p,q,L,\Lambda_0,|\cO|$ and $T$, by induction, we get estimate \eqref{estim thm max principle global 1}.
\end{proof}

\begin{thm}\label{thm max princip global 2}
  Let assumptions $({\mathcal A}1)$--$({\mathcal A}4)$ be satisfied and  $(u,v)\in\cV_{2,0}(Q)\times\cM^2(Q)$ be a weak solution of   (\ref{1.1}). If
  $L_0=0$ and with probability 1
  \begin{equation}\label{condition in thm max princip 2}
    \begin{split}
      f(t,x,R,0,0)\equiv f(t,x,0,0)
       \textrm{ and }
       g(t,x,R,0,0) \textrm{ are decreasing in }  R\in \bR
    \end{split}
  \end{equation}
  for all $(t,x)\in [0,T]\times\bR^n$,
  then we assert
  \begin{equation}\label{estim in thm max princip global 2}
    \begin{split}
        \esssup_{(\omega,t,x)\in\Omega\times Q}u(\omega,t,x)
        \leq   \esssup_{(\omega,t,x)\in\Omega\times \partial_{\rm p}Q}u^+(\omega,t,x)
        +C A_p(f_0,g_0^+) |\cO|^{\frac{1}{n+2}-\frac{1}{p}}
    \end{split}
  \end{equation}
  with the constant $C$ only depending on
  $n,p,q,\kappa,\lambda,\beta,\varrho,T,\Lambda_0$ and $L$.
\end{thm}

\begin{proof}[{\bf Proof}]
  We use  De Giorgi iteration and the same notations
  in the proof of Theorem \ref{thm max princip global 1}. Similar to the proof of \eqref{eq 1 in thm glob max}
  and \eqref{eq 2 in thm glob max},
  under condition \eqref{condition in thm max princip 2},
  we have for each $t\in [0,T]$,
  \begin{equation}\label{eq 1 in thm glob max 2}
      \begin{split}
        &\esssup_{\Omega}\sup_{\tau\in[t,T]}2E\left[\int_{\tau}^T\ll (u(s)-k)^+,\,\   g_0^k (s)\gg ds   \big|\sF_{\tau}\right]\\
        \leq &
        \esssup_{\Omega}\sup_{\tau\in[t,T]}2E\left[\int_{\tau}^T\ll (u(s)-k)^+,\,\ g_0(s)\gg ds   \big|\sF_{\tau}\right]\\
        \leq &
        \delta\|(u-k)^+\|^2_{\frac{2(n+2)}{n};\cO_t}
        +C(\delta)\left|A_p(f_0,g_0^+)\right|^2
        \left|\{u>k\}\right|_{\infty;\cO_t}^{1-\frac{2}{p}}
      \end{split}
    \end{equation}
    and almost surely
    \begin{equation}\label{eq 2 in thm glob max 2}
      \begin{split}
        &2E\left[\int_t^T\ll |\nabla (u(s)-k)^+|,\,| f_0^k (s)|\gg ds\Big| \sF_t\right]\\
        =&
        2E\left[\int_t^T\ll |\nabla (u(s)-k)^+|,\, |f_0(s)|\gg ds \Big| \sF_t\right]\\
        \leq &
        \eps E\left[\int_t^T\|\nabla (u(s)-k)^+\|^2_{L^2(\cO)}\, ds \Big| \sF_t\right]
        +C(\eps)\left|A_p(f_0,g_0^+)\right|^2
        \left|\{u>k\}\right|_{\infty;\cO_t}^{1-\frac{2}{p}}.
      \end{split}
    \end{equation}
    Hence instead of \eqref{eq 5 in thm glob max}, we obtain
    \begin{equation*}
      \psi(h)\leq \frac{C\left|A_p(f_0,g_0^+)\right|^{\alpha}}{(h-k)^{\alpha}}\psi(k)^{1+\bar{\eps}}.
    \end{equation*}
    By Corollary \ref{cor degiorgi iteration}, for any $\bar{\theta}_0\geq C A_p(f_0,g_0^+)|\cO_{t_1}|^{\frac{1}{n+2}-\frac{1}{p}}$,
    we have
    \begin{equation}
      \begin{split}
        \left|\left\{u>\esssup_{(\omega,t,x)\in\Omega\times \partial_{\rm p}Q}u^+(\omega,t,x)
        +\bar{\theta}_0    \right\}\right|_{\infty;\cO_{t_1}}=0,
      \end{split}
    \end{equation}
    which implies
    \begin{equation}
      \begin{split}
        \sup_{(\omega,t,x)\in\Omega\times \cO_{t_1}}u
        \,\leq
        \sup_{(\omega,t,x)\in\Omega\times \partial_{\rm p}Q}u^++CA_p(f_0,g_0^+)|\cO_{t_1}|^{\frac{1}{n+2}-\frac{1}{p}}
      \end{split}
    \end{equation}
    where the constant $C$ depends only on $n,\lambda,p,q,\beta,\kappa,\varrho,\Lambda_0$ and $L$.
    As $T-t_1$ only depends on the structure terms, by induction, we get estimate \eqref{estim in thm max princip global 2} where the constant $C$ also depends on $T$. We complete the proof.
\end{proof}

\begin{rmk}
  In Theorem \ref{thm max princip global 2},   we can dispense with the assumptions that $L_0=0$ and  the function
  $r \mapsto g(t,x,r,0,0)$ decreases in $r$, by considering the function
   $\tilde{u}:=e^{2(L+L_0)t}u(t,x)$ instead of $u$.
\end{rmk}

\begin{cor}
  Let assumptions $({\mathcal A}1)$--$({\mathcal A}4)$  be satisfied with $L_0=0$.
  Let the two pair $(f,g^1)$ and $(f,g^2)$  satisfy condition \eqref{condition in thm max princip 2} in Theorem \ref{thm max princip global 2}.
  Assume that $G^1$ and $G^2$ are two random variables in $L^{\infty}(\Omega,\sF_T,L^2(\cO))$.
  Let $(u_i,v_i)\in \sU\times\sV(G^i,f,g^i)$, $i=1,2$ and  $(u_1-u_2)^+\in \dot{\cV}_{2,0}(Q)$. Then
  if $G^1\leq G^2$ $dP\otimes dx$-a.e. and $g^1(\omega,t,x,u_2,\nabla u_2,v_2)\leq g^2(\omega,t,x,u_2,\nabla u_2,v_2), dP\otimes dt\otimes dx$-a.e.,  we have $u_1(\omega,t,x)\leq u_2(\omega,t,x)$, $dP\otimes dt\otimes dx$-a.e..
\end{cor}
\begin{proof}
  $(u_1-u_2,v_1-v_2)$ belongs to $\sU\times\sV(\tilde{G},\tilde{f},\tilde{g})$ with
  \begin{equation*}
    \begin{split}
      \tilde{f}(s,x,R,Y,Z):=\, & f(s,x,R+u_2(s,x),Y+\nabla u_2(s,x),Z+v_2(s,x))\\
      &-f(s,x,u_2(s,x),\nabla u_2(s,x),v_2(s,x)),\\
      \tilde{g}(s,x,R,Y,Z):=\, &g^1(s,x,R+u_2(s,x),Y+\nabla u_2(s,x),Z+v_2(s,x))\\
      &-g^2(s,x,u_2(s,x),\nabla u_2(s,x),v_2(s,x))
    \end{split}
  \end{equation*}
  and  $\tilde{G}:=G^1-G^2$. Since $\tilde{G}\leq 0$, $\tilde{g}_0\leq 0$ and $f_0=0$, the assertion follows from Theorem \ref{thm max princip global 2}.
\end{proof}

\subsection{The local case}
This subsection is devoted to the local  regularity of weak solutions.

\begin{defn}
  For domain $Q'\subset Q$, a function $\zeta(\cdot,\cdot)$ is called a cut-off function on $Q'$ if

  (i) $\zeta\in \dot{W}_1^{2,2}(Q')$, i.e. there exists a sequence $\{\zeta^l,l\in\bN\}\subset C_c^{\infty}(Q')$ such that
  \begin{equation}
    \begin{split}
      \|\zeta^l-\zeta\|_{W_1^{2,2}}:=
      \bigg\{\int_{Q'}\Big(& |(\zeta^l-\zeta)(t,x)|^2+|\partial_{t}(\zeta^l-\zeta)(t,x)|^2\\
      &+
        |\nabla (\zeta^l-\zeta)(t,x)|^2+|\nabla^2(\zeta^l-\zeta)(t,x)|^2\Big)\,dxdt \bigg\}^{\frac{1}{2}}
            \end{split}
  \end{equation}
  converges to zero as $l$ tends to infinity with $\nabla^2(\zeta^l-\zeta)(t,x)$ being the Hessian matrix of the function $(\zeta^l-\zeta)(t, \cdot)$ at $x$;

  (ii) $0\leq \zeta \leq 1 $;

  (iii) there exists a  domain  $ Q''\Subset Q'$ and a nonempty domain $Q'''\Subset Q''$ such that
     \begin{equation*}
    \zeta(t,x)=
  \left\{\begin{array}{l}
    \begin{split}
        &1,\quad (t,x)\in Q''',\\
        &0,\quad (t,x)\in Q'\setminus Q'';
    \end{split}
  \end{array}\right.
  \end{equation*}

  (iv) $|\nabla \zeta|,\partial_{t}\zeta \in L^{\infty}(Q')$.

  For simplicity, we denote
  $$ \|\nabla \zeta\|_{L^{\infty}(Q')}:=\| |\nabla \zeta|  \|_{L^{\infty}(Q')}. $$
\end{defn}

First, to study the local behavior of our weak solutions, we shall generalize the deterministic parabolic De Giorgi class (c.f. \cite{ChenyazheParabolic,Ladyzhenskaia_68,Lieberman,Wangguanglei_Harnack}) to our stochastic version and introduce the definition of
 De Giorgi class in the backward stochastic parabolic case.
 \begin{defn}
   We say that a function $u\in\cV_{2,0}(Q)$ belongs to the backward stochastic parabolic De Giorgi class (BSPDG, for short)
   if for any $k\in\bR$, $Q_{\rho,\tau}:=[t_0-\tau,t_0)\times B_{\rho}(x_0) \subset Q$ (with $\rho,\tau\in (0,1)$)
   and any cut-off function $\zeta$ on $Q_{\rho,\tau}$, we have
   \begin{equation*}
  \begin{array}{l}
    \begin{split}
           &\|\zeta(u-k)^{\pm}\|^2_{\cV_{2}(Q_{\rho,\tau})}\\
       \leq
       &\gamma\Big\{\|(u-k)^{\pm}\|^2_{2;Q_{\rho,\tau}}
       (1+\|\nabla\zeta\|_{L^{\infty}(Q_{\rho,\tau})}^2
            +\|\partial_{t}\zeta\|_{L^{\infty}(Q_{\rho,\tau})}  )\\
       &\quad +(k^2+a_0^2)|\{ (u-k)^{\pm}>0\}|_{\infty;Q_{\rho,\tau}}^{1-\frac{2}{\mu}}   \Big\}
    \end{split}
  \end{array}
  ~~~~~~~(\mathfrak{D}^{\pm})
  \end{equation*}
  for some triplet $(a_0, \mu, \gamma) \in [0, \infty)\times (n+2, \infty)\times [0, \infty)$.
  We  call $a_0,\mu,$ and $\gamma$ the structural parameters of $BSPDG^{\pm}$. We mean that $u\in\cV_{2,0}(Q)$
  satisfies $(\mathfrak{D}^{+})$ ($(\mathfrak{D}^-)$, respectively) by the inclusion $u\in BSPDG^+(a_0, \mu, \gamma;Q)$
  ($u\in BSPDG^-(a_0, \mu, \gamma;Q)$, respectively). We say $u\in BSPDG(a_0, \mu, \gamma;Q)$ if both inclusions $u\in BSPDG^+(a_0, \mu, \gamma;Q)$
  and $u\in BSPDG^-(a_0, \mu, \gamma;Q)$ hold.
 \end{defn}

 \begin{prop}\label{prop verifying BSPDG}
   Let assumptions $({\mathcal A}1)$--$({\mathcal A}3)$  hold. Assume that $(u,v)\in\cV_{2,0}(Q)\times\cM^2(Q)$
   is a weak solution of   (\ref{1.1}). Then we assert that $u\in BSPDG(a_0, \mu, \gamma;Q)$,
   with $a_0:=A_p(f_0,g_0), \mu :=\min\{p,2q\}$,  and some parameter $\gamma $   depending on
  $n,p,q,\kappa,\lambda,\beta,$ $\varrho,\Lambda,\Lambda_0$ and $L$.
 \end{prop}

\begin{rmk}
  It is worth noting that in this proposition, assumption $({\mathcal A}4)$ is not made.
\end{rmk}

 The proof requires the following lemma.

\begin{lem}\label{lem cor of prop ito non null bdary}
     Let assumptions $({\mathcal A}1)$--$({\mathcal A}3)$ hold, $\zeta$ be a cut-off function
     on $Q_{\rho,\tau}:=[t_0-\tau,t_0)\times B_{\rho}(x_0) \subset Q$, and $(u,v)\in\cV_{2,0}(Q)\times\cM^2(Q)$
     be a weak solution of   (\ref{1.1}). Then, we have almost surely
\begin{equation}\label{eq ito cutoff func}
  \begin{split}
    &\ll \zeta^2(t),\, |u^+(t)|^2\gg_{B_{\rho}(x_0)}+\int_t^{t_0}\ll \zeta^2(s),\,|v^u(s)|^2\gg_{B_{\rho}(x_0)} ds\\
    =&
%    \int_{B_{\rho}(x_0)}\zeta^2|u^+|^2(t_0,x)dx
    -\int_t^{t_0}2\ll \zeta \partial_{s}\zeta(s),\ \, |u^+(s)|^2\gg_{B_{\rho}(x_0)} ds\\
    &+\int_t^{t_0}2\ll \zeta^2(s) u^+(s),\ \,g^{u}(s)\gg_{B_{\rho}(x_0)}ds               \\
    &+\int_t^{t_0}2\ll \zeta^2(s) u^+(s),\ \,   b^i(s)\partial_{x_i}u(s)+c(s)\,u^+(s)+\varsigma^r(s) v^{r,u}(s)\gg_{B_{\rho}(x_0)} ds\\
    &-\int_{t}^{t_0}\ll 2\partial_{x_i} (\zeta^2(s)u^+(s)),\ \,
            a^{ji}(s)\partial_{x_j}  u^+(s)
            +\sigma^{ir}(s)v^{r,u}(s)
            +f^{i,u}(s)\gg_{B_{\rho}(x_0)}ds\\
    &-\int_t^{t_0}2\ll \zeta^2(s)u^+(s),\ \, v^{r,u}(s)\gg_{B_{\rho}(x_0)} dW_s^r, \quad \forall \, t\in [t_0-\tau, t_0]
  \end{split}
\end{equation}
where
 \begin{equation*}
    \begin{split}
        &g^{u}(s,x):=1_{\{(s,x):u(s,x)>0\}}(s,x)g(s,x,u(s,x),\nabla u(s,x),v(s,x));\\
        &f^{i,u}(s,x):=1_{\{(s,x):u(s,x)>0\}}(s,x)f^i(s,x,u(s,x),\nabla u(s,x),v(s,x)),\ i=0,1,\cdots,n;\\
    \end{split}
  \end{equation*}
  and
 $$ v^u:=(v^{1,u},\cdots,v^{m,u}), \quad v^{r,u}(s,x):=1_{\{(s,x):u(s,x)>0\}}(s,x)v^r(s,x),\quad  r=1,\cdots,m.$$
\end{lem}
The proof of this lemma is rather standard and is sketched below.
\begin{proof}  [{\bf Sketch of the proof}]
We use approximation.
    By the definition of a cut-off function, all terms of \eqref{eq ito cutoff func} are well defined and there is a sequence $\{\zeta^l,l\in\bN\}\subset C_c^{\infty}(Q_{\rho,\tau})$ such that $\lim_{l\rightarrow\infty}\|\zeta^l-\zeta\|_{W^{2,2}_1(Q_{\rho,\tau})}=0$.
    In view of Definition \ref{definition of weak solution} and Remark \ref{rmk defn weak solution}, we verify like in \textbf{Step 1} of the proof of Lemma \ref{lem ito for non null bdary condition} that for each $l$ there holds
    \begin{equation*}
      \begin{split}
        \zeta^l u(t,x)
        =&
            \int_t^T \Big[
            \partial_{x_j}\left(a^{ij}\partial_{x_i}(\zeta^l u)(s,x)
            +\sigma^{jr}\zeta^l v^r(s,x)+ \tilde{f}_l^j(s,x)\right)+b^i\partial_{x_j}(\zeta^l u)(s,x)
        \\
        &
            +c\,\zeta^l u(s,x)+\varsigma^{r}\zeta^lv^r(s,x)
            +\tilde{g}_l(s,x)\Big]\,ds
            -\int_t^T \zeta^lv^r(s,x)\,dW_s^r,\quad t\in [0,T]
      \end{split}
    \end{equation*}
    in the weak sense of Definition \ref{definition of weak solution}, where
    \begin{equation*}
      \begin{split}
        \tilde{g}_l(s,x):=&-\partial_{s}\zeta^lu(s,x)
        +\zeta^l(s,x) g(s,x,u(s,x),\nabla u(s,x),v(s,x))\\
        &-b^i\partial_{x_i}\zeta^lu(s,x)
        -\partial_{x_j} \zeta^l \bar{f}_l^j(s,x),\\
        \bar{f}_l(s,x):=\, &a^{i\cdot} \partial_{x_i}  u(s,x)
        +\sigma^{\cdot r} v^r(s,x)+ f(s,x,u(s,x),\nabla u(s,x),v(s,x)), \\
        \tilde{f}_l(s,x):=\, &-a^{i\cdot} \partial_{x_i}\zeta^l u(s,x)+\zeta^l(s,x)f(s,x,u(s,x),\nabla u(s,x),v(s,x)).
      \end{split}
    \end{equation*}
    Thus, $(\zeta^lu,\zeta^lv)\in \dot{\sU}\times\dot{\sV}(0,\tilde{f}_l,\tilde{g}_l)$.
    From  Proposition \ref{propsition ito non null bdary}  we conclude  that
      \eqref{eq ito cutoff func} holds with $\zeta$ being replaced by $\zeta^l$. Passing to the limit in $L^{1}(\Omega\times Q)$ and taking into account  the path-wise continuity of $u$, we prove our assertion.
\end{proof}

 \begin{proof}[{\bf Proof of Proposition \ref{prop verifying BSPDG}}]
   Consider the cylinder
   $$Q_{\rho,\tau}(X)=X+[-\tau,0)\times B_\rho(0)\subset Q \hbox{ \rm with } X:=(t_0,x_0).$$
    For simplicity, we denote $Q_{\rho,\tau}(X)$ and $B_{\rho}(x_0)$ by $Q_{\rho,\tau}$ and $B_{\rho}$ respectively. Let $\zeta$ be a cut-off function on $Q_{\rho,\tau}$. Denote $\bar{u}:=(u-k)^+$. From Lemma \ref{lem cor of prop ito non null bdary}, it follows that
   \begin{equation}\label{eq 1 in prop verifying}
     \begin{split}
       &E\bigg[ \|\zeta(t)\bar{u}(t)\|_{L^2(B_{\rho})}^2
       +\int_t^{t_0}\|\zeta(s) v_k(s)\|_{L^2(B_{\rho})}^2\,ds   \big|\sF_t\bigg]\\
       =&
%       E\bigg[ \int_{B_{\rho}}\zeta^2\bar{u}^2(t_0,x)dx
%            \big|\sF_t\bigg]+
       E\bigg[ \int_t^{t_0}2\ll \zeta^2(s)\bar{u}(s),\,  g^k  (s,\cdot,\bar{u}(s),\nabla\bar{u}(s),v_k(s))
       \gg_{B_{\rho}}  ds
            \big|\sF_t\bigg]\\
       &+E\left[ \int_t^{t_0} 2\ll
             \zeta^2(s)\bar{u}(s),\,
            b^i(s)\partial_{x_i}u(s)+c(s)\,\bar{u}(s)+\varsigma^r(s) v^{r}_k(s)\gg_{B_{\rho}}
       \big|\sF_t\right]\\
       &-E\bigg[ \int_t^{t_0}2\ll \zeta(s)\partial_{s}\zeta(s),\, |\bar{u}(s)|^2\gg_{B_{\rho}} ds    \big|\sF_t\bigg]\\
       &-E\bigg[ \int_t^{t_0}\ll 2\partial_{x_j} (\zeta^2(s)\bar{u}(s)),\,
       a^{ij}(s)\partial_{x_i} \bar{u}(s)+\sigma^{jr}(s)v^r_k(s)\\
       &~~~~~~~~~~~~~~~~~~~
       +(f^k)^j(s,\cdot,\bar{u}(s),\nabla \bar{u}(s),v_k(s))   \gg_{B_{\rho}} ds
            \big|\sF_t\bigg]
     \end{split}
   \end{equation}
   holds almost surely  for all $t\in[t_0-\tau,t_0)$ where $v_k:=v1_{u>k}$ and
   for  $(\omega,t,x,R,Y,Z)\in \Omega\times [t_0-\tau,t_0)\times\cO\times\bR\times\bR^n\times\bR^m $
  $$ ( f^k  , g^k  )(\omega,t,x,R,Y,Z):=(f,g)(\omega,t,x,R+k,Y,Z)+(0,c(\omega,t,x)k).
    $$

   In view of \eqref{eq est g}-\eqref{eq est principle part} and \eqref{eq 1 in thm glob max}-\eqref{eq 4 in thm glob max}, we have almost surely for all $t\in [t_0-\tau,t_0)$
   \begin{equation*}
     \begin{split}
       &E\bigg[ \int_t^{t_0}2\ll \zeta^2\bar{u}(s),\,  g^k  (s,\cdot,\bar{u}(s),\nabla\bar{u}(s),v_k(s))
       \gg_{B_{\rho}} ds
            \big|\sF_t\bigg]\\
       \leq &\
       E\bigg[ \int_t^{t_0}2\ll \zeta^2\bar{u}(s),\,  g_0^k  (s)+L(|\bar{u}(s)|
        +|\nabla \bar{u}(s)|
        +|v_k(s)|)\gg_{B_{\rho}} ds  \big|\sF_t\bigg]\\
       \leq &
       \ \eps_1\|\zeta\bar{u}\|^2_{\cV_2(Q_{\rho,\tau})}
       +\eps_2 E\bigg[ \int_t^{t_0} \|\zeta(s)v_k(s)\|^2_{L^2(B_{\rho})}\,ds   \big| \sF_t \bigg]
       \\
      &+C(L)\|\nabla \zeta\|_{L^{\infty}(Q_{\rho,\tau})}^2 \|\bar{u}\|^2_{2;Q_{\rho,\tau}}
      +C(\eps_1,L,n,p)(\left|A_p(f_0,g_0)\right|^2+k^2)|\{u>k \}|^{1-\frac{2}{p}}_{\infty;Q_{\rho,\tau}}
        \\
       &+
       C(\eps_1,\eps_2,L)
       \|\zeta\bar{u}\|^2_{2;Q_{\rho,\tau}}
       +E\bigg[ \int_t^{t_0} \!2\!\ll \zeta^2\bar{u}(s),\, |c(s)k| \gg_{B_{\rho}} ds   \big| \sF_t \bigg],
    \end{split}
   \end{equation*}

   \begin{equation*}
     \begin{split}
       &E\bigg[ \int_t^{t_0} 2\ll \zeta^2\bar{u}(s),\,\, |c(s)k| \gg_{B_{\rho}} ds   \big| \sF_t \bigg]\\
       \leq &E\bigg[ \int_t^{t_0}\ll \zeta^2\left|\bar{u}(s)\right|^2,\ |c(s)|\gg_{B_{\rho}} ds \big| \sF_t\bigg]
       +k^2 E\bigg[ \int_t^{t_0}|\! \ll \zeta^2(s),\ |c(s) 1_{u>k}|\gg_{B_{\rho}}  ds \big| \sF_t\bigg]\\
       \leq&\ E\bigg[ \int_t^{t_0}\! \ll \zeta^2\left|\bar{u}(s)\right|^2,\  |c(s)|\gg_{B_{\rho}} ds \big| \sF_t\bigg]
       +k^2\Lambda_0 |\{u>k\}|_{\infty;Q_{\rho,\tau}}^{1-\frac{1}{q}},
     \end{split}
   \end{equation*}

   \begin{equation*}
    \begin{split}
   &E\left[ \int_t^{t_0} 2\ll
             \zeta^2(s)\bar{u}(s),\,
            b^i\partial_{x_i}u(s)+c\,\bar{u}(s)+\varsigma^r v^{r}_k(s)\gg_{B_{\rho}}
       \big|\sF_t\right]\\
   \leq &
    \,\eps_1 \|\zeta\bar{u}\|^2_{\cV_2(Q_{\rho,\tau})}
    +\eps_2  \|\zeta v_k\|^2_{2;Q_{\rho,\tau}}
    +\eps_1\|\nabla \zeta\|_{L^{\infty}(Q_{\rho,\tau})}^2 \|\bar{u}\|^2_{2;Q_{\rho,\tau}}\\
    &\ +C(\eps_1,\eps_2,n)
    E\bigg[ \int_t^{t_0} \ll \zeta^2\bar{u}^2(s),\ |b(s)|^2+|c(s)|+|\varsigma(s)|^2 \gg_{B_{\rho}}    ds   \big| \sF_t \bigg]\\
    \leq &
    \,\eps_1 \|\zeta\bar{u}\|^2_{\cV_2(Q_{\rho,\tau})}
    +\eps_2  \|\zeta v_k\|^2_{2;Q_{\rho,\tau}}
    +\eps_1\|\nabla \zeta\|_{L^{\infty}(Q_{\rho,\tau})}^2 \|\bar{u}\|^2_{2;Q_{\rho,\tau}}
    +C(\eps_1,\eps_2,n)\Lambda_0 \|\zeta\bar{u}\|_{\frac{2q}{q-1};Q_{\rho,\tau}}^2\\
    \leq &
    \,2\eps_1 \|\zeta\bar{u}\|^2_{\cV_2(Q_{\rho,\tau})}
    +\eps_2  \|\zeta v_k\|^2_{2;Q_{\rho,\tau}}
    +\eps_1\|\nabla \zeta\|_{L^{\infty}(Q_{\rho,\tau})}^2 \|\bar{u}\|^2_{2;Q_{\rho,\tau}}\\
    &+C(\eps_1,\eps_2,n,q,\Lambda_0) \|\zeta\bar{u}\|_{2;Q_{\rho,\tau}}^2\\
    \end{split}
   \end{equation*}
   and
   \begin{equation*}
     \begin{split}
       &-2E\bigg[ \int_t^{t_0}\!\!\!\ll \partial_{x_j} (\zeta^2\bar{u}(s))
       ,\, a^{ij}\partial_{x_i} \bar{u}(s)+\sigma^{jr}v^r_k(s)+(f^k)^j(s,\bar{u}(s),\nabla \bar{u}(s),v_k(s))   \gg_{B_{\rho}}\! ds
            \big|\sF_t\bigg]\\
       &
        =-2E\bigg[ \int_t^{t_0}\!\!\!\!\ll \zeta^2\partial_{x_j} \bar{u}(s)
        ,\, a^{ij}\partial_{x_i} \bar{u}(s)+\sigma^{jr}v^r_k(s)+(f^k)^j(s,\bar{u}(s),\nabla \bar{u}(s),v_k(s))   \gg_{B_{\rho}}\! ds
            \big|\sF_t\bigg]\\
       &
        \ -4 E\bigg[ \int_t^{t_0}\!\!\!\!\ll \bar{u}\zeta \partial_{x_j}\zeta(s)
        ,\, a^{ij}\partial_{x_i} \bar{u}(s)+\sigma^{jr}v^r_k(s)
        +(f^k)^j(s,\bar{u}(s),\nabla \bar{u}(s),v_k(s))   \gg_{B_{\rho}}\! ds
            \big|\sF_t\bigg]\\
      &
      \leq   -\lambda_0 E\bigg[ \int_t^{t_0}\|\zeta\nabla \bar{u}(s)\|^2_{L^2(B_{\rho})}\,ds
                            \big|\sF_t\bigg]
        +\alpha_0 E\bigg[ \int_t^{t_0}\|\zeta v_k(s)\|^2_{L^2(B_{\rho})}\,ds
                            \big|\sF_t\bigg]
      \\
      &
        +C\left(\|\zeta \bar{u}\|^2_{2;Q_{\rho,\tau}}+
        \left(\left|A_p(f_0,g_0)\right|^2+k^2\right)|\{u>k\}|_{\infty;Q_{\rho,\tau}}^{1-\frac{2}{p}}
        \right)
      \\
      &
        +CE\left[
            \int_t^{t_0}\ll |\bar{u}\nabla \zeta(s)|,\, |f^k_0(s)|+|\bar{u}(s)|
            + |\zeta\nabla\bar{u}(s)|
            +|\zeta v_k(s)| \gg ds\big|\sF_t
        \right]
      \\
      &
        \leq    -(\lambda_0-\eps) E\bigg[ \int_t^{t_0}\|\nabla (\zeta\bar{u}(s))\|^2_{L^2(B_{\rho})}\,ds
                            \big|\sF_t\bigg]
        +(\alpha_0+\eps) E\bigg[ \int_t^{t_0}\|\zeta v_k(s)\|^2_{L^2(B_{\rho})}\,ds
                            \big|\sF_t\bigg]
                                                                            \\
       &
       +C\left\{\|\nabla\zeta\|^2_{L^{\infty}(Q_{\rho,\tau})}  \|\bar{u}\|^2_{2;Q_{\rho,\tau}}+\|\bar{u}\zeta\|^2_{Q_{\rho,\tau}}+ \left(\left|A_p(f_0,g_0)\right|^2+k^2\right)|\{u>k\}|_{\infty;Q_{\rho,\tau}}^{1-\frac{2}{p}}
       \right\}
     \end{split}
   \end{equation*}
    with $C:=C(\eps,p,\lambda,\beta,\varrho,\kappa,\Lambda,L)$,
     where $\alpha_0\in (0,1)$ and $\lambda_0$ are two positive constants depending only on structure terms such as $\kappa,p,\lambda,\varrho,\beta,\Lambda$ and $L$, and the three parameters $\eps,\eps_1,\eps_2$ are waiting to  be determined later.
   On the other hand, it is obvious that almost surely
   $$
   -E\left[ \int_t^{t_0}2\ll
        \zeta\partial_s\zeta(s),|\bar{u}(s)|^2
        \gg_{B_{\rho}} ds  \Big|\sF_t\right]
   \leq 2 \|\partial_s \zeta\|_{L^{\infty}(Q_{\rho,\tau})}\|\bar{u}\|^2_{2;Q_{\rho,\tau}},\ \,\forall t\in[t_0-\tau,t_0).
      $$

    Therefore, combining the above estimates and \eqref{eq 1 in prop verifying} and choosing the parameters $\eps,\eps_1$ and $\eps_2$ to be small enough, we obtain
   \begin{equation*}
    \begin{array}{l}
        \begin{split}
           &\|\zeta(u-k)^{+}\|^2_{\infty,2;Q_{\rho,\tau}}+ \|\nabla (\zeta(u-k)^{+})\|^2_{2;Q_{\rho,\tau}}\\
       \leq &
       \gamma\Big\{
       (1+\|\nabla\zeta\|_{L^{\infty}(Q_{\rho,\tau})}^2
            +\|\partial_{t}\zeta\|_{L^{\infty}(Q_{\rho,\tau})}  )\|(u-k)^{+}\|^2_{2;Q_{\rho,\tau}}\\
        &\quad +\left(k^2+\left|A_p(f_0,g_0)\right|^2\right)|\{ (u-k)^{+}>0\}|_{\infty;Q_{\rho,\tau}}^{1-\frac{2}{p\wedge (2q)}}   \Big\}
        \end{split}
    \end{array}
    \end{equation*}
    where $\gamma$ is a positive constant depending on the structure terms
    such as $n,p,q,\kappa,\lambda,\varrho,\beta, L,\Lambda$ and $\Lambda_0$. Hence $u\in BSPDG^+(a_0, \mu, \gamma;Q)$.

    In a similar way, we show $u\in BSPDG^-(a_0, \mu, \gamma;Q)$. The proof is complete.
 \end{proof}

\begin{thm}\label{thm max princip local}
  If $u\in BSPDG^{\pm}(a_0, \mu, \gamma;Q)$, we assert that for any
  $$ Q_{\rho}=[t_0,t_0+\rho^2)\times B_{\rho}(x_0) \subset Q,\,\rho\in (0,1), $$
  there holds
  \begin{equation}\label{estim in thm max princip local}
    \esssup_{\Omega\times Q_{\frac{\rho}{2}}} u^{\pm}
    \,\leq
        C\left\{  {\rho^{-\frac{n+2}{2}}} \|u^{\pm}\|_{2;Q_{\rho}}+a_0 \rho^{1-\frac{n+2}{\mu}}   \right\},
  \end{equation}
  where $C$ is a constant depending only on $a_0, \mu, \gamma$ and $n$.
\end{thm}

\begin{proof}[{\bf Proof}]
Consider $u\in BSPDG^+(a_0, \mu, \gamma;Q)$.
  Take
  $$R_l=\frac{\rho}{2}+\frac{\rho}{2^{l+1}},\,k_l=k(2-\frac{1}{2^l}),\,l=0,1,2,\cdots  $$
  where $k$ is a parameter waiting to be determined later.
  Denote $Q^l:=Q_{R_l}=[t_0,t_0+R_l^2)\times B_{R_l}(x_0) $. Choose $\zeta_l$ to be a cut-off function on $Q^l$ such that
  \begin{equation*}
    \zeta_l(t,x)=
  \left\{\begin{array}{l}
    \begin{split}
        &1,\quad (t,x)\in Q^{l+1};\\
        &0,\quad (t,x)\in Q^l \setminus Q_{\frac{R_l+R_{l+1}}{2}}
    \end{split}
  \end{array}\right.
  \end{equation*}
  and
  $$\left\|\nabla \zeta_l  \right\|^2_{L^{\infty}(Q_{\rho})}+\left\| \partial_{t}\zeta_l  \right\|_{L^{\infty}(Q_{\rho})}
  \leq \frac{C(n)}{(R_l-R_{l+1})^2}.$$

  From $(\mathfrak{D}^+)$, it follows that
  \begin{equation*}
    \begin{split}
      &\|\zeta_l (u-k_{l+1})^+\|^2_{\cV_2(Q^l)}\\
      \leq \,\,&
        C2^{2l}{\rho^{-2}}\|(u-k_{l+1})^+\|^2_{2;Q^l}
        +C(k^2+a_0^2)|\{u>k_{l+1}\}|^{1-\frac{2}{\mu}}_{\infty;Q^l}.
    \end{split}
  \end{equation*}
  For $k\geq a_0\rho^{1-\frac{n+2}{\mu}}$, we obtain from Lemma~\ref{lem emmbedding for space V} that
  \begin{equation*}
    \begin{split}
      &\|\zeta_l (u-k_{l+1})^+\|^2_{\frac{2(n+2)}{n};Q^l}\\
      \leq\  &
        C\|\zeta_l (u-k_{l+1})^+\|^2_{\cV_2(Q^l)}\\
      \leq \ &
      C2^{2l}{\rho^{-2}}\|(u-k_{l+1})^+\|^2_{2;Q^l}
      + Ck^2 {\rho^{-2(1-\frac{n+2}{\mu})}}|\{u>k_{l+1}\}|^{1-\frac{2}{\mu}}_{\infty;Q^l}.
    \end{split}
  \end{equation*}
  Setting
  $$ \phi_l:=\|(u-k_l)^+\|^2_{2;Q^l},  $$
we have
  \begin{equation*}
    \begin{split}
      \phi_{l+1}
      \leq\ & \|\zeta_l(u-k_{l+1})^+\|_{2;Q^l}^2\\
      \leq\  & |\{u>k_{l+1}\}|^{\frac{2}{n+2}}_{\infty;Q^l}\|\zeta_l(u-k_{l+1})^+\|^2_{\frac{2(n+2)}{n};Q^l}
      ~\textrm{  (H$\ddot{\textrm{o}}$lder inequality)}\\
      \leq\  &
      C2^{2l}{\rho^{-2}}\phi_l |\{u>k_{l+1}\}|^{\frac{2}{n+2}}_{\infty;Q^l}
      + Ck^2 {\rho^{-2(1-\frac{n+2}{\mu})}}|\{u>k_{l+1}\}|^{1-\frac{2}{\mu}+\frac{2}{n+2}}_{\infty;Q^l}.
    \end{split}
  \end{equation*}
  Note that
  \begin{equation*}
    \begin{split}
      \phi_l
      =\|(u-k_l)^+\|^2_{2;Q^l}
      \geq\ &
      (k_{l+1}-k_l)^2 |\{u>k_{l+1}\}|_{\infty;Q^l}
      = k^2{2^{-(2l+2)}} |\{u>k_{l+1}\}|_{\infty;Q^l}.
    \end{split}
  \end{equation*}
  Hence,
  \begin{equation*}\label{eq a in thm max princip local}
    \begin{split}
      \phi_{l+1}
      \leq\ &
       C2^{2l(1+\frac{2}{n+2})} \left[
      {\rho^{-2} k^{-\frac{4}{n+2}}}{\phi_l^{1+\frac{2}{n+2}}}
      +
      {\rho^{-2(1-\frac{n+2}{\mu})}k^{\frac{4}{\mu}-\frac{4}{n+2}}}
      {\phi_l^{1-\frac{2}{\mu}+\frac{2}{n+2}}}
      \right]\\
      =\ &
        C2^{2l(1+\frac{2}{n+2})}
        {\rho^{-2(1-\frac{n+2}{\mu})}k^{\frac{4}{\mu}-\frac{4}{n+2}}}
        {  \phi_l^{1-\frac{2}{\mu}+\frac{2}{n+2}}}
      \left[\left({k^{-2} \rho^{-n-2}}{\phi_l}\right)^{\frac{2}{\mu}}  +1
      \right] .
    \end{split}
  \end{equation*}

  For $k\geq a_0\rho^{1-\frac{n+2}{\mu}}+ {\rho^{-\frac{n+2}{2}}}\|u^+\|_{2;Q_{\rho}}$,
  we have ${k^{-2}\rho^{-n-2}}{\phi_l} \leq 1$
  and therefore
  \begin{equation*}\label{eq b in thm max princip local}
    \begin{split}
      \phi_{l+1} \leq
      C 2^{2l(1+\frac{2}{n+2})}
      {\rho^{-2(1-\frac{n+2}{\mu})}k^{\frac{4}{\mu}-\frac{4}{n+2}}}
      { \phi_l^{1-\frac{2}{\mu}+\frac{2}{n+2}}}
      .
    \end{split}
  \end{equation*}
  Setting
  $$
  \alpha_l:=\rho^{-n-2}k^{-2}{\phi_l},
  $$
  we have
  $$
  \alpha_{l+1}\leq C_1 2^{2l(1+\frac{2}{n+2})} \alpha_l^{1-\frac{2}{\mu}+\frac{2}{n+2}}.
  $$
  From Lemma \ref{lem degiorgi iteration}, we see that the following
  \begin{equation*}
    \begin{split}
      \alpha_0 = &{k^{-2}\rho^{-n-2}}{\|(u-k)^+\|^2_{2;Q_{\rho}}}
      \ \leq\
      {k^{-2}\rho^{-n-2}}{\|u^+\|^2_{2;Q_{\rho}}}
      \ \leq\
      \theta_0: = C_1^{-\frac{1}{\alpha}}2^{-\frac{2}{\alpha^2}(1+\frac{2}{n+2})}
    \end{split}
  \end{equation*}
  with $\alpha:=\frac{2}{n+2}-\frac{2}{\mu}$, implies
  $$\lim_{l\rightarrow\infty}\alpha_l=0\textrm{ and thus }\lim_{l\rightarrow\infty}\phi_l=0.$$

  In conclusion, the two inequalities
  $$k^2\geq {\theta_0}^{-1} \rho^{-n-2}  {\|u^+\|_{2;Q_{\rho}}^2}
 \textrm{ and }
 k\geq a_0\rho^{1-\frac{n+2}{\mu}}+ \rho^{-\frac{n+2}{2}}\|u^+\|_{2;Q_{\rho}}$$
  imply the following one:
  \begin{equation}\label{eq c in thm max princip local}
    \begin{split}
      \|u\|_{\infty;Q_{\frac{\rho}{2}}}\leq 2k.
    \end{split}
  \end{equation}
  Hence, \eqref{eq c in thm max princip local} holds for the following choice
  $$k:=a_0 \rho^{1-\frac{n+2}{\mu}}+\left(1+{\theta_0^{-\frac{1}{2}}}\right) {\rho^{-\frac{n+2}{2}}}{\|u^+\|_{2;Q_{\rho}}}.$$
  which implies our desired estimate.

  For $u\in BSPDG^-(a_0, \mu, \gamma;Q)$, the desired assertion follows in a similar way. We complete our proof.
\end{proof}

\bibliographystyle{siam}
%\bibliography{ref_qjn}

\end{document}